\documentclass[onefignum,onetabnum]{siamart171218}

\ifpdf
\hypersetup{
  pdftitle={mean-field limit of interacting 2D nonlinear stochastic spiking neurons},
  pdfauthor={B. Aymard, F. Campillo and R. Veltz}
}
\fi

\usepackage[utf8]{inputenc} 
\usepackage[T1]{fontenc} 
\usepackage{amsmath,amssymb}
\usepackage{graphicx}
\usepackage{todonotes}
\usepackage{color}
\usepackage{algorithmic}    

\usepackage{xfrac}
\usepackage{hyperref}
\hypersetup{colorlinks=true,citecolor=blue}

\usepackage{ifthen}
\newboolean{DisplayREM}

\graphicspath{{Images/}}

\newcommand{\rem}[2][noinline]{\ifthenelse{\boolean{DisplayREM}}{\todo[#1, color=gray!20!white]{\small \texttt{Rem}: #2}}{}}

\usepackage{pifont}

\newcommand{\NN}{\mathcal{N}}
\newcommand{\VV}{\mathcal{V}}
\newcommand{\WW}{\mathcal{W}}
\newcommand{\CC}{\mathcal{C}}
\newcommand{\LL}{\mathcal{L}}

\newcommand{\LLt}{\mathcal{L}^{\textrm{\upshape\tiny t}}}
\newcommand{\LLj}{\mathcal{L}^{\textrm{\upshape\tiny j}}}

\newcommand{\LLtt}{\mathcal{L}^{\textrm{\upshape\tiny t}*}}
\newcommand{\LLjt}{\mathcal{L}^{\textrm{\upshape\tiny j}*}}

\newcommand{\R}{\mathbb{R}}

\newcommand{\E}{\mathbb{E}}

\newcommand{\crochet}[1] {\left\langle #1 \right\rangle}
\newcommand{\bigcrochet}[1] {\bigl\langle #1 \bigr\rangle}
\newcommand{\norm}   [1] {\left\Vert #1 \right\Vert}

\newcommand{\rmd}  {{\textrm{\upshape d}}}
\newcommand{\indic}{{\mathrm\mathbf1}}
\newcommand{\bv}{{\bar v}}
\renewcommand{\div}{\textrm{\upshape div}}
\newcommand{\Omegad}   {\Omega^{\textrm{\upshape\tiny d}}}

\newcommand{\bbracket}[1]{
  \bigl[\!\!\bigl[#1\bigr]\!\!\bigr]}

\newcommand{\mA}{\mathbf A}
\newcommand{\mB}{\mathbf B}
\newcommand{\mD}{\mathbf D}

\reversemarginpar

\newcommand\defeq{\stackrel{\smash{\scriptscriptstyle\mathrm{def.}}}{=}}

\newtheorem{remark}{Remark}

\newcommand{\vertiii}[1]{{\left\vert\kern-0.25ex\left\vert\kern-0.25ex\left\vert #1 
    \right\vert\kern-0.25ex\right\vert\kern-0.25ex\right\vert}}

\title{Mean-field limit of interacting 2D nonlinear stochastic spiking neurons}

\author{Benjamin Aymard\thanks{Team MathNeuro, Inria Sophia-Antipolis M\'editerrann\'ee, 06902 Sophia-Antipolis Cedex, France 
  (\email{romain.veltz@inria.fr}).}
\and Fabien Campillo\footnotemark[1]
\and Romain Veltz\footnotemark[1]}

\setboolean{DisplayREM}{false}
\begin{document}
\maketitle

\begin{abstract}
In this work, we propose a nonlinear stochastic model of a network of stochastic spiking neurons.
We heuristically derive the mean-field limit of this system.
We then design a Monte Carlo method for the simulation of the microscopic system, and a finite volume method (based on an upwind implicit scheme) for the mean-field model.
The finite volume method respects numerical versions of the two main properties of the mean-field model, conservation and positivity, leading to existence and uniqueness of a numerical solution.
As the size of the network tends to infinity, we numerically observe propagation of chaos and convergence from an individual description to a mean-field description.
Numerical evidences for the existence of a Hopf bifurcation (synonym of synchronised activity) for a sufficiently high value of connectivity, are provided.
\end{abstract}


\section{Introduction}

Bridging the gap between microscopic and macroscopic descriptions of biological neural networks is one of the current challenges in neuroscience.
Microscopic description of individual neurons has been widely studied, from the historical model of Louis Lapicque \cite{Lapicque1907} to the Nobel winning model of Hodgkin and Huxley \cite{HodgkinHuxley1952}, reproducing reality in a very satisfying way, both qualitatively and quantitatively.
On the other end of the spectrum, electroencephalography  \cite{Berger1929} and functional magnetic resonance imaging \cite{fMRI1990} provide insight in brain activity and connectivity at a macroscopic level.
But how can we explain the macroscopic measurements from the microscopic description?
A direct approach consists of considering a large model coupling single neuron models in a finite size network description.
The complexity of this approach increases with the size of the system, leading to a high computational cost for simulations, and a difficult theoretical analysis.
An alternative approach is to use the mean-field approximation, where the individual description is homogeneous stemming from exchangeability of the particles.
An advantage of the mean-field description is the possibility to use the tools of PDEs theory for theoretical analysis.
Also, in this approach, the complexity remains constant, independent of the number of neurons in the system.

The family of two-dimensional (2D) nonlinear spiking neuron model \cite{gerstner_neuronal_2014, Touboul2008,touboul_dynamics_2008} is efficient at reproducing the majority of observed membrane potential behaviours.
It is a good compromise between the simple but limited integrate-and-fire model, and the rich but highly complex Hodgkin-Huxley model.
However, it has been shown that networks composed of such neurons, where the spiking process is deterministic, might lead to a blow up of the solution in finite time \cite{Delarue2015}.
The model introduced recently in \cite{DeMasi2015, Fournier2016}, where the spiking event is stochastic, elegantly circumvents this problem.
Additionally, stochastic spiking with a probability increasing with the membrane potential value is biologically relevant \cite{FiringVariability2011} and has been used \cite{Schwalger2017} to model cortical columns.
Yet, the neuron model they introduced in \cite{DeMasi2015} is only 1D, limiting the dynamics to integrator neurons without  bursting nor resonance for instance.

The use of 2D nonlinear spiking neurons in a network and the study of the mean-field approximation brings many difficulties which have been tackled only recently~\cite{diVolo2019}. The first difficulty is that spikes are modelled as finite time explosions rather than threshold crossing. Note that although the singular behaviour of \cite{Delarue2015} might disappear for this neuron model, we focus on the framework of \cite{DeMasi2015}. The second difficulty is that the rate function which governs the spiking mechanism is unbounded which renders the use of classical theory for piecewise deterministic Markov processes \cite{Davis1993} more difficult. Note that a related network with scalar neurons, \textit{e.g.} the quadratic IF neuron, is widely studied in the literature under the name of the Ott-Antonsen ansatz \cite{Ott2008,Pazo2014,Pazo2016}.

The study of similar noiseless case networks have been studied in the past \cite{nicola_nonsmooth_2016,nicola_bifurcations_2013} under the first order moment closure approximation which effectively makes the system one-dimensional. However, this approximation implies that one loses information regarding the adaptation variable. Nevertheless, the authors were able to compute bifurcation diagrams thanks to this approximation and to predict bursting activity at the network level. Recently, several approximations have been re-visited \cite{Augustin2017} with similar aims. Note that network bursts have also been predicted recently based on separation of timescales approximation \cite{fardet_understanding_2018}.

Please note that this mean-field approach for spiking neural networks has been extensively used in the past \cite{Brunel2000,renart_mean-field_2004,ostojic_synchronization_2009} and put on rigorous grounds in \cite{carrillo_classical_2013,Delarue2015}. The use of generalised integrate and fire neurons in mean-field models has been studied in \cite{Gerstner2014,DeMasi2015,Fournier2016}. The mean-field limits of networks of spiking neurons modelled by Hawkes processes has been studied by \cite{chevallier_microscopic_2015}. Several extensions have been provided namely by introducing space dependency \cite{duarte_hydrodynamic_2015} or dendritic compartments \cite{inglis_mean-field_2015,fournier_toy_2018}.

From the simulation point of view, finite size network dynamics might be efficiently computed using a Monte Carlo approach and taking advantage of graphics processing unit (GPU), whereas conservative mean-field may be simulated using finite volume methods (FVM).

The PDEs obtained as mean-field limits of networks of 2D nonlinear spiking neuron are challenging to simulate (see \cite{kamps_computational_2019}) and to study theoretically:
the velocity field can be explosive (one may think about adaptive exponential integrate and fire models \cite{BretteGerstner2005}),
the reset introduces a spatial singularity in the form of a Dirac  measure,
and the network activity includes a nonlocal term.
Qualitative properties of FVM are crucial in order to reliably approximate such equations (see \cite{marpeau_finite_2009} for a diffusive case), especially when the aim is to capture invariant measures.

In this article, we introduce a 2D nonlinear stochastic spiking neural network combining \cite{Touboul2008} and \cite{DeMasi2015}.
In  \cref{sec.model}, we present the model and heuristically derive its mean-field limit.
\Cref{sec.numerics} is devoted to numerical methods and their properties.
Finally, in \cref{sec.simulations} we present numerical simulations of the model.

\section{A detailed stochastic model and its mean-field approximation}
\label{sec.model}

\subsection{Detailed stochastic model}
\label{sec.stochastic.model}

We consider a population of $N$ identical neurons.
Let  $v_i(t) \in \mathbb{R}$ (resp. $w_i(t) \in \mathbb{R}$) denote the membrane potential (resp. the adaptation current) of neuron $i$ at time $t$.
The orders of magnitude are typically millisecond for the time, millivolt for membrane potential and pico Amp\`ere for the adaptation current.
The adaptation current is not directly linked to a biological quantity, but represents internal processes such as ion channel dynamics or propagation of depolarisation.

The model we consider is a piecewise deterministic Markov process (PDMP).
The flow is given by
\begin{equation}
\label{micro}
\begin{aligned}
  \dot v_i(t) 
  &= 
  \tilde\VV(v_i(t),w_i(t))\,,
  &&
  \textrm{with }
  &\tilde\VV(v,w)
  &\defeq F(v) - w  + I\,,
\\
  \dot w_i(t) 
  &= 
  \WW(v_i(t),w_i(t))\,, 
  &&
  \textrm{with }
  &\WW(v,w)
  &\defeq\frac{1}{\tau_w} (b \,v - w)
\end{aligned}
\end{equation}
where $\tau_w$ and $b$ are real parameters, and $I$ corresponds to the external stimuli (input current).
Several choices are possible for the nonlinearity $F$, corresponding to different classical models, such as
\[
\begin{aligned}
F_1(v) &= v(v  - a), a\in \mathbb{R} &&\mbox{               (Izhikevich model \cite{Izhikevich2004})},   \\
F_2(v) &= e^v - v  &&\mbox{             (AdEx model \cite{BretteGerstner2005})},\\
F_3(v) &= v^4 + {2\,a v}, a\in \mathbb{R} &&\mbox{          (quartic model \cite{Touboul2008})}.
\end{aligned}
\]

Coupled with \cref{micro} and independent of each other, each neuron $i$ spikes at a given rate $\lambda(v^i)> 0$ depending on the membrane potential, and imposes a jump transition given by
\begin{equation}
\begin{aligned}
  \bigl(v_i(t),w_i(t)\bigr) &= \bigl(\bar{v},w_{i}(t^-)+{\bar w}\bigr),\\
  \bigl(v_j(t),w_j(t)\bigr) &= \textstyle \bigl(v_j(t^-) + \frac{J}{N},w_j(t^-)\bigr) \mbox{ for }j \not= i\,.
\end{aligned}
\label{spike}
\end{equation}
where $\bar w>0$. Note that is is possible to adapt, at minor cost, the algorithms that follows to the case $\bar w<0$.

The first part of the transition describes the reset following the emission of a spike by an isolated neuron while the second part models the excitatory interaction between neurons. In short, when a neuron spikes, it increases the membrane potentials of the post-synaptic neurons by the amount $\frac JN$. We focus on the case $J>0$ although $J<0$ can be treated analogously.

We assume that the mechanisms acting on each neuron are independent and identically distributed at initial time according to $\mu_0(\rmd v,\rmd w)$; and that the spike events for every neuron $i$ and the initial condition of the system are mutually independent.

This stochastic approach was originally introduced in \cite{gerstner_neuronal_2014,DeMasi2015} in the framework of 1D linear spiking neurons.
Stochastic firing is biologically justified \cite{FiringVariability2011}, and avoids a mathematical problem appearing within networks using deterministic firing \cite{Delarue2015}.
In this work, we generalise this approach to the case of 2D nonlinear spiking neurons.

\subsection{mean-field approximation}

Let  $(v^N_{i}(t),w^N_{i}(t))_{1\leq i\leq N}$ denote the dynamic \cref{micro}--\cref{spike} described in the previous section. It can be rewritten as the following interacting particle system:
\begin{align}
\label{eq.pdmp}
\left\{
\begin{aligned}
  v^N_{i}(t)
  &=
  v^N_{i,0}
  + \int_{0}^t \tilde\VV(v^N_{i}(s),w^N_{i}(s))\,\rmd s 
\\
  &\qquad\qquad\qquad\quad
  +
  \int_{0}^t\int_{0}^\infty
  (\bar v - V^N_{i}(s^-)) \,\indic_{\{z\leq \lambda(v^N_{i}(s^-))\}}\,\NN^{i}(\rmd z,\rmd s)   
\\
  &\qquad\qquad\qquad\quad
  +
  \frac{J}{N}
  \sum_{j\neq i} \int_{0}^t\int_{0}^\infty \,\indic_{\{z\leq \lambda(v^N_{j}(s^-))\}}\, \NN^{j}(\rmd z,\rmd s)   \,,
\\ 
  w^N_{i}(t)
  &=
  w^N_{i,0}
  +
  \int_{0}^t \WW(v^N_{i}(s),w^N_{i}(s))\,\rmd s 
\\
  &\qquad\qquad\qquad\quad
  +
  \int_{0}^t\int_{0}^\infty \bar w \,\indic_{\{z\leq \lambda(v^N_{i}(s^-))\}} \,\NN^{i}(\rmd z,\rmd s)\,,
\end{aligned}
\right.
\end{align}
where 
$\NN^{i}$ are $N$ independent Poisson random measures with intensity measure $\rmd z\times \rmd s$ (Lebesgue measure on $\R^2_{+}$). We suppose that the initial conditions $(v^N_{i,0},w^N_{i,0})$ are independent with the same distribution $\mu_{0}$, for all $i$, and that they are also independent from the $\NN^{i}$'s. Particles in \eqref{eq.pdmp} interact only through the term:
\[
  \Psi^{N}\!(t)
  \defeq
   \frac{1}{N}
  \sum_{j\neq i} \int_{0}^t\int_{0}^\infty \,\indic_{\{z\leq \lambda(v^N_{j}(s^-))\}}\, \NN^{j}(\rmd z,\rmd s)\,.
\] 
Suppose that the interacting particle system \eqref{eq.pdmp}  features a propagation of chaos property \cite{sznitman1991a};  this will have to be demonstrated in a future study. This property, that we do not prove in the present work, roughly means that when $N$ is large, the particles tend to behave like independent particles with the same limit distribution. Hence, according to the law of large numbers, the (stochastic) interaction term $\Psi^N{t}$ converges to the following (deterministic) expression:
\[
  \Psi^N\!(t)
  \underset{\scriptscriptstyle N \mathrm{ large}}{\simeq}
  \Psi(t)
  \defeq
  \E \int_{0}^t 
         \int_{0}^\infty \,\indic_{\{z\leq \lambda(v^N_{1}(s^-))\}}\, \NN^{1}(\rmd z,\rmd s)
\]
as $\NN^{1}$ is a Poisson random measure with intensity measure $\rmd z\times \rmd s$, the limit interaction term is:
\[
  \Psi(t)
  =
  \E\int_{0}^t
      \int_{0}^\infty \,\indic_{\{z\leq \lambda(v^N_{1}(s^-))\}}\,\rmd z\,\rmd s
  =
  \int_{0}^t\psi(s) \, \rmd s
  \quad \textrm{ with }\psi(t)\defeq \E(\lambda(v(t))) \,.
\]
In other words, we have a mean-field limit, indeed the empirical distribution of the interacting particles:
\begin{align}
\label{eq.muN}
\mu^N (t,\rmd v,\rmd w)
&\defeq 
\frac{1}{N}\sum_{i=1}^N \delta_{(v^N_{i}(t),w^N_{i}(t))}(\rmd v,\rmd w)
\end{align}
converges to a (deterministic) distribution 
$\mu(t,\rmd v, \rmd w)$ which represents the distribution of a ``limit'' particle $(v(t),w(t))$ described as:
\begin{align}
\label{eq.McKean-Vlasov}
\left\{
\begin{aligned}
  v(t)
  &=
  v_{0}
  + \int_{0}^t \tilde\VV(v(s),w(s))\,\rmd s 
\\
  &\qquad\qquad\qquad
  +
  \int_{0}^t\int_{0}^\infty
  (\bar v - v(s^-)) \,\indic_{\{z\leq \lambda(v(s^-))\}}\,\NN(\rmd z,\rmd s)   
\\
  &\qquad\qquad\qquad
  +
  J
  \int_{0}^t\E(\lambda(v(s))) \, \rmd s   \,,
\\ 
  w(t)
  &=
  w_{0}
  +
  \int_{0}^t \WW(v(s),w(s))\,\rmd s 
\\
  &\qquad\qquad\qquad
  +
  \int_{0}^t\int_{0}^\infty \bar w \,\indic_{\{z\leq \lambda(v(s^-))\}} \,\NN(\rmd z,\rmd s)
\end{aligned}
\right.
\end{align}
where $(v_{0},w_{0}) \sim \mu_{0}$ and $\NN(\rmd z,\rmd s)$ is a random Poisson measure with intensity measure $\rmd z\times\rmd s$, $(v_{0},w_{0})$ and $\NN(\rmd z,\rmd s)$ are independent.
 
\subsection{Derivation of the nonlinear PDE}

The  PDMP  \cref{eq.McKean-Vlasov} is of McKean-Vlasov type as it is not simply an equation for $(v(t),w(t))$ \emph{per se} as the right hand side of \cref{eq.McKean-Vlasov} depends on $(v(t),w(t))$ and on its law through the deterministic current $\psi(t)=\E(\lambda(v(t)))$. The infinitesimal generator of this process is:
\[
  \LL_{\mu(t)} \varphi(v,w)
  =
  \LLt_{\mu(t)} \varphi(v,w)
  +
  \LLj \varphi(v,w)\,,
\]
with
\begin{align}
\label{eq.LLt}
   \LLt_{\mu} \varphi(v,w) 
   &\defeq
   \VV_{\mu}(v,w)\,\frac{\partial \varphi(v,w)}{\partial v}
   +\WW(v,w)\,\frac{\partial \varphi(v,w)}{\partial w}\,,
 \\
\label{eq.LLj}
   \LLj \varphi(v,w) 
   &\defeq
   \lambda(v) \iint_{\R^2} \bigl(\varphi(v',w')-\varphi(v,w)\bigr)\,\pi(v,w,\rmd v',\rmd w')
 \\
\nonumber
   &\qquad\qquad=
   \lambda(v) \bigl(\varphi(\bar v,\bar w+w)-\varphi(v,w)\bigr)\,,
\end{align}
and
\begin{align*}
  \VV_{\mu}(v,w)
  &\defeq
  \tilde\VV(v,w)+J\,\iint_{\R^2} \lambda(v')\,\mu(\rmd v',\rmd w')\,,
\\
  \pi(v,w,\rmd v',\rmd w')
  &
  \defeq
  \delta_{\bar v}(\rmd v')\,\delta_{\bar w+w}(\rmd w')  \,.
\end{align*}
The generator is defined for all test functions $\varphi:\R^2\mapsto \R$ in $\CC^1_{b}(\R^2)$ (continuously differentiable in $(v,w)$ and bounded). The operator $\LLt_{\mu}$ corresponds to the ``ODE'' part of the McKean-Vlasov PDMP   \cref{eq.McKean-Vlasov} and $\LLj$ to the ``pure jump'' part.

\medskip

The evolution of the distribution $\mu(t)$ is given by the Kolmogorov forward equation as a weak PDE:
\begin{align}  
\label{eq.transport.weak}
  \frac{\rmd}{\rmd t}\crochet{\mu(t),\varphi}
  &
  =\crochet{\mu(t),\LL_{\mu(t)}\varphi}\,,
  \quad\textrm{ for } t> 0\textrm{ and }
  \crochet{\mu(0),\varphi}=\crochet{\mu_{0},\varphi}
\end{align} 
for all test functions, where $\crochet{\mu,\varphi} \defeq \iint_{\R^2}\varphi(v,w)\,\mu(\rmd v, \rmd w)$ is the usual duality bracket. Note that the adjoint $\LL_{\mu}^*$ of $\LL_{\mu}$ is given by:
\begin{align*}  
  \crochet{\LL_{\mu}^*\mu, \varphi}
  &=
  \crochet{\mu,\LL_{\mu} \varphi}
\\
  &=
  \crochet{\mu,\VV_{\mu}\,\partial_v\varphi+\WW\,\partial_w\varphi}
  + \crochet{\delta_{\bar v}\otimes
      \textstyle\int_{\R} \lambda(v')\,\sigma_{\bar w}\mu(t,\rmd v',\cdot),\varphi}
  - \crochet{\mu,\lambda\, \varphi}
\\
  &=
  \crochet{\LLtt_{\mu}\mu+\LLjt\mu, \varphi}\,,
\end{align*} 
with $\sigma_{\bar w}\mu(t,\rmd v,\rmd w) \defeq \mu(t,\rmd v,\rmd w-{\bar w})$, so that \cref{eq.transport.weak} is a weak form of the following strong form of the forward Kolmogorov equation:
\begin{align}
\label{eq.kolmo.strong}
   \frac{\partial}{\partial t}\mu(t)
   =
   \LL^*_{\mu(t)}\,\mu(t)
   =
   \bigl(\LLtt_{\mu(t)}+\LLjt\bigr)\,\mu(t)
\end{align}  
which reads:
\begin{align}  
\label{macro}
  &\frac{\partial}{\partial t}  \mu(t,v,w)
  +
  \frac{\partial}{\partial v}\bigl(\mu(t,v,w)\,\VV_{\mu(t)}(v,w)\bigr)
  +
  \frac{\partial}{\partial w}\bigl(\mu(t,v,w)\,\WW(v,w)\bigr)
\\
\nonumber
  &\qquad\qquad
  =
  -\lambda(v)\,\mu(t,v,w)
  +
  \delta_{\bar v}(v) \int_{\R} \lambda(v')\,\sigma_{\bar w}\mu(t,v',w)\,\rmd v'
  \,,\quad\mu(0)=\mu_{0}
\end{align} 
where the shift operator $\sigma_{\bar w}$ is now
$\sigma_{\bar w}\mu(t,v,w) \defeq \mu(t,v,w-{\bar w})$.

\subsection{The mean-field as coupled transport equations }
The PDE \cref{macro} can be represented as a transport PDE with a jump condition at the interface:
\begin{align*} 
  \Gamma &\defeq \bigl\{(\bar v,w)\,;\, w\in\R  \bigr\}\,.
\end{align*} 
Indeed, let:
\begin{align*} 
  \Omega_1 &\defeq \bigl\{(v,w)\in\R^2, v<\bar v  \bigr\}\,,
  &
  \Omega_2 &\defeq \bigl\{(v,w)\in\R^2,  v>\bar v  \bigr\}\,.
\end{align*} 
In \cref{eq.transport.weak} consider the term:
\begin{align*} 
&\bigcrochet{\mu(t),\nabla \varphi \cdot F_{\mu(t)}}
=
\textstyle
\iint_{\R^2} \nabla \varphi \cdot \bigl(\mu(t)\, F_{\mu(t)}\bigr)
\\
&\quad=
\textstyle
\iint_{\Omega_{1}} \nabla \varphi \cdot \bigl(\mu(t) \,F_{\mu(t)}\bigr)
+
\iint_{\Omega_2} \nabla \varphi \cdot \bigl(\mu(t) \,F_{\mu(t)}\bigr)
\\
&\quad=
\textstyle
-\iint_{\Omega_{1}} \varphi \,\div \bigl(\mu(t) \,F_{\mu(t)}\bigr)
+ \int_{\R} \varphi(\bv,w)\,\mu(t,\bv^-,w)\,F_{\mu(t)}(\bv,w)\cdot
          \bigl(\begin{smallmatrix}1\\ 0\end{smallmatrix}\bigr)\,\rmd w
\\
&
\textstyle
\qquad -\iint_{\Omega_2}  \varphi \, \div \bigl(\mu(t)\, F_{\mu(t)}\bigr)
+ \int_{\R} \varphi(\bv,w)\,\mu(t,\bv^+,w)\,F_{\mu(t)}(\bv,w)\cdot
          \bigl(\begin{smallmatrix}-1\\ 0\end{smallmatrix}\bigr)\,\rmd w
\\
&\quad=
\textstyle
-\iint_{\R^2} \varphi \,\div \bigl(\mu(t) \,F_{\mu(t)}\bigr)
- \int_{\R} \varphi(\bv,w)\,\bbracket{\mu(t,\cdot,w)}\,\VV_{\mu(t)}(\bv,w)\,\rmd w
\end{align*}
where
\[ 
  \bbracket{\mu(t,\cdot,w)}  \defeq \mu(t,\bv^+,w)-\mu(t,\bv^-,w)
\]
denotes the jump of  $\mu(t,v,w)$ through the interface $\Gamma$.
Using test functions $\varphi$ with support included in $\Omega_{1}\cup\Omega_{2}$ we get:
\begin{align}
\label{strong.edp} 
    &\partial_{t} \mu(t)
  +
  \partial_{v}\bigl(\mu(t)\,\VV_{\mu(t)}\bigr)
  +
  \partial_{w}\bigl(\mu(t)\,\WW\bigr)
  =
  -\lambda \,\mu(t)\,,
\\
\nonumber
  &\qquad\qquad\qquad\textrm{with }\mu(t=0) = \mu_{0}\,,\hbox{ on }\Omega_i\,,\ i=1,2\,,
\end{align}
where $F_{\mu(t)} = \bigl(\VV_{\mu(t)}\,,\,\WW\bigr)$;
and with test functions $\varphi$ with support included in $\Gamma$ we get the interface condition:
\begin{align}
\label{strong.edp.interface}
   \VV_{\mu(t)}(\bv,w)\,\bbracket{\mu(t,\cdot,w)}
   =
   \int_{\R}\lambda(v')\,\sigma_{\bar w}\mu(t,v',w)\,\rmd v'\,,\ 
   \forall w\in\R\,.
\end{align}
Hence the PDE \cref{strong.edp} with the interface condition \cref{strong.edp.interface} is equivalent to \cref{macro}. Note that the PDE \cref{strong.edp} can be solved using the method of characteristics and that the boundary conditions are specified if and only if the characteristics enter the domain.


\section{Numerical methods}
\label{sec.numerics}

\subsection{Monte Carlo simulation of the particle system}

We now pre\-sent a Monte Carlo simulation procedure of the interacting particles system  introduced in Section \ref{sec.stochastic.model} and rewritten as \cref{eq.pdmp}. The Monte Carlo procedure is exact up to the time discretisation of the trajectories.

For the finite size network description, we take advantage of the all-to-all connectivity of the network to design a simple Monte Carlo algorithm.
Let  $v_i^n$ (resp. $w_i^n$) denote the membrane potential (resp. the adaptation current) of neuron $i$ at time $t^n$.
After initialisation on each neuron $i$:
\[
  (v_i^0,w_i^0) \sim \mu_0\,,
\]
the method consists of three steps.
First, we use an Euler step for the ODE integration to predict the deterministic evolution on each neuron $i$:
\begin{align*}
  v_i^* &= v_i + \Delta t\, \tilde\VV(v_i^n,w_i^n) \,,
  &
  w_i^* &= \textstyle w_i + \Delta t\, \WW(v_i^n,w_i^n)\,,
  &
  i=1,\dots,N\,.
\end{align*}
Second, we assume that the jump rate $\lambda$ does not evolve much in between the jump times.
Hence, neuron $i$ fires with probability $\lambda(v_i^*)\,\Delta t$:
\[
\begin{aligned}
  b_{i} &= \indic_{u_{i} < \lambda(v_i^*)\,\Delta t}\,, \textrm{ with } u_i \sim U(0,1)\,, \\
  v_i^{**} &= (1-b_{i})\,v_i^* + b_{i}\,\bar{v}\,,\\
  w_i^{n+1} &= w_i^* + b_{i}\,{\bar w}
\end{aligned}
\]
independently on each neuron $i$.
Third, the sum of all the spikes:
\[
  s = \frac{J}{N} \sum_{i=1}^N b_i
\]
is distributed to each neuron separately:
\[
v_i^{n+1} = v_i^{**} + s \,, \qquad i=1,\dots,N\,.
\]

This method is stable under certain conditions.
For the deterministic part of the PDMP, the classical Courant–Friedrichs–Lewy condition for forward Euler must hold.
For the probabilistic part, the time step has to be small enough so that, for each neuron, at most one jump occurs during the time step.
This depends on the parameters of the model.
In practice, we use the time step for the flow part and plot at the same time as the jump frequency, and then reduce the time step if necessary.

\medskip

The exact Monte Carlo methods that are developing more and more in the field of neuroscience \cite{Veltz2015,lemaire2018a}. Here, in the case of a simple Euler scheme for the flow approximation, the proposed ``almost'' exact Monte Carlo method is very effectively simulated on graphics processing units (GPU).

\subsection{Simulation of the mean-field PDE \cref{macro}}

As this PDE is defined on $\mathbb{R}^2$, a preliminary step is to consider a compact subset:
\[
  \Omega \defeq [v_{\min},v_{\max}] \times [w_{\min},w_{\max}]
\]
large enough to contain the initial condition and the solution during the evolution.

The numerical method for the simulation of the macroscopic description \cref{macro} is based on a splitting strategy \cite{Strang1968}:
we solve alternatively the transport part, using a conservative finite volume method (FVM, see \cite{Vinokur1989}), and the source part, using the approximate solution of an ODE system.
We define a regular $N_v\times N_w$ grid on $\Omega$:
\[
  \Omegad 
  \defeq 
  \bigl\{(v_i,w_j) \,,\, 1\leq i\leq N_{v}\,,\ 1\leq j\leq N_{w} \bigl\}
\]
with
\[
  (v_i,w_j) 
  \defeq 
  \bigl(v_{\min}+ (i-1)\,\Delta v\,,\,w_{\min} + (j-1)\,\Delta w\bigr)
\]
with $\Delta v>0$ and $\Delta w>0$ being the mesh size in the $v$ and $w$ direction respectively, and suppose $v_{\max}=v_{N_{v}}$,  $w_{\max}=w_{N_{w}}$.
Define also $v_{i\pm1/2}=v_{i}\pm\Delta v/2$, $w_{j\pm1/2}=v_{j}\pm\Delta w/2$, and:
\begin{align*}
  \bar{i} 
  &\defeq 
  \lfloor \textstyle \frac{{\bar v}-v_{\min}}{\Delta v} \rfloor\,,
&
  \bar{j}
  &\defeq
  \lfloor \textstyle \frac{{\bar w}-w_{\min}}{\Delta w} \rfloor\,.
\end{align*}
Cells are defined by
\begin{align*}
  \Omega_{i,j} 
  &= 
  \bigl[v_i-\Delta v/2,v_i+ \Delta v/2\bigr] 
  \times 
  \bigl[w_j-\Delta w/2,w_j+\Delta w/2\bigr]
\end{align*}
for $1\leq i\leq N_{v}$, $1\leq j\leq N_{w}$.

Let $\Delta t>0$ denote the time step, and $t^n=n\,\Delta t$.
We now propose a finite volume  approximation
$\mu^n=(\mu^n_{i,j})_{1\leq i\leq N_{v},1\leq j\leq N_{w}}$  of $\mu(t)$ at time $t^n$:
\begin{align*}
  \mu^n_{i,j} 
  &\simeq 
  \frac{1}{\Delta v\,\Delta w}
  \crochet{\mu(t^n),\indic_{\Omega_{i,j}}}\,,
\end{align*}
with initial condition $\mu^0_{i,j} \simeq \crochet{\mu_{0},\indic_{\Omega_{i,j}}}/(\Delta v\,\Delta w)$.

\medskip

For the update $\mu^n$ to $\mu^{n+1}$, we adopt a splitting technique for~\cref{macro} that will be detailed later. This technique alternates numerical approximations 
of the ``transport'' part:
\begin{align}
\label{eq.split.1}
   \partial_{t} \rho(t) = (\LLt_{\rho(t)})^* \rho(t)\,,
\end{align}
and of the ``jump'' part:
\begin{align}
\label{eq.split.2}
   \partial_{t} \rho(t) = (\LLj)^* \rho(t)\,.
\end{align}
We first detail the approximation of these two PDEs and then the splitting technique.

\subsection*{Discretisation of transport part}

We describe the update $\mu\to\mu'$ which corresponds to the 
numerical approximation of \cref{eq.split.1} on the interval $[0,\Delta t]$ with initial condition $\mu$ and final value at $\Delta t$ assigned to $\mu'$.
We adopt an upwind finite volume scheme 
on the structured mesh $\Omegad$ for the space and a semi-implicit Euler scheme for the time:
\begin{align}
\label{eq.split.discret.t}
  \frac{\mu'- \mu}{\Delta t} 
  +
  \mD(\mu)\,\mu'
  =
  0.
\end{align}
$\mD(\mu)\,\mu'$ is the approximation of divergence operator 
 $ \partial_{v}(\mu'\,\VV_{\mu})
  +
  \partial_{w}(\mu'\,\WW)
$
given by
\[
  \bigl(\mD(\mu) \, \mu'\bigr)_{i,j} 
  \defeq
  \frac{1}{\Delta v} \, \bigl(F_{i+1/2,j} - F_{i-1/2,j}\bigr) 
  + 
  \frac{1}{\Delta w} \, \bigl(G_{i,j+1/2} - G_{i,j-1/2}\bigr)
\]
based on the upwind numerical fluxes defined by\rem[]{On remonte les caracteristiques}
\begin{equation}
\begin{aligned}
  F_{i+1/2,j} 
  &\defeq
  \begin{cases}
    V_{i+1/2,j}(\mu)\,\mu'_{i,j}\,, &\mbox{ if } V_{i+1/2,j}(\mu)>0\\
    V_{i+1/2,j}(\mu)\,\mu'_{i+1,j}\,, &\mbox{ otherwise},
  \end{cases}
\\
  G_{i,j+1/2} 
  &\defeq
  \begin{cases}
    W_{i,j+1/2}\,\mu'_{i,j}\,, &\mbox{ if } W_{i,j+1/2}>0\\
    W_{i,j+1/2}\,\mu'_{i,j+1}\,, &\mbox{ otherwise},
  \end{cases}
\end{aligned}
\label{flux}
\end{equation}
(see  \cite{Vinokur1989})
and the approximated vector field, using numerical integration, by:
\[
\begin{aligned}
  V_{i,j} (\mu)
  &= 
  F(v_i) - w_j + I 
  + J \sum_{i'=1}^{N_v} \sum_{j'=1}^{N_w} 
  f(v_{i'})\,\mu_{i',j'} 
     \,\Delta v \,\Delta w\,,
\\
  W_{i,j} &= {(b\,v_i - w_j)}/{\tau_w}\,.
\end{aligned}
\]
Transition \eqref{eq.split.discret.t} reads:
\begin{equation}
  \mA_{\Delta t}(\mu) \, \mu' 
  = \mu
\label{FVM}
\end{equation}
with
\begin{equation*}
  \mA_{\Delta t}(\mu)
  \defeq 
  I + \Delta t \,\mD(\mu) \,.
\end{equation*}

Note that $\mA_{\Delta t}(\mu)$ and $\mD(\mu)$ can be considered as $(N_v\times N_w)^2$ square (sparse) matrices, but to avoid notational complexities, we will consider them as linear operators on the set of real functions defined on $\Omegad$. 

Concerning the boundary conditions, we impose null fluxes on the boundaries, leading to
\rem[inline]{with Boundary Conditions and Ghost Cells leveque2002a p. 129}
\begin{equation}
    \label{BCnumerics}
  \begin{cases}
    \forall n\geq 0\,,\ 1 \leq j \leq N_w\,:\   F_{1/2,j} = F_{N_v+1/2,j} = 0 \,,
    \\
    \forall n \geq 0\,,\ 1\leq i \leq N_v \,:\  G_{i,1/2} = G_{i,N_w+1/2} = 0\,. 
  \end{cases}
\end{equation}
This semi-implicit approach allows to avoid the restrictive Courant–Friedrichs–Lewy  condition that would appear due to the strong nonlinearity $F$ (possibly exponential).
Also, the matrix of the linear system to solve at each iteration is sparse, because the nonlocal term is treated explicitly.

\begin{remark}
	It is easy to see that $\mD$ is composed of a matrix diagonal by bands with 5 bands at $-N_v,-1,0,1,N_v$. This is a direct consequence of the formulas \cref{flux} based on a classical stencil scheme: each cell is (possibly)  connected to its 4 direct neighbours plus itself. The same holds for $\mD^*$.
\end{remark}

\subsection*{Discretisation of the jump part}

We describe the update $\mu\to\mu'$ that corresponds to the numerical approximation of \cref{eq.split.2} on the interval $[0,\Delta t]$ with initial condition $\mu$ and final value at $\Delta t$ assigned to $\mu'$.
Equation \cref{eq.split.2} is the forward Kolmogorov equation of a pure jump Markov process with infinitesimal generator $\LLj$. First, we restrict the latter operator to $\Omega$. The most evident way to achieve this is to consider:
$   \lambda(v)\,\bigl(\varphi(\bar v,(\bar w +w)\wedge w_{\max})-\varphi(v,w)\bigr),
$
which amounts to an accumulation point at $(\bar v,w_{\max})$. This operator is then 
approximated by a jump process on the grid $\Omegad$ by considering the generator:
\begin{align}
\label{eq.tilde.Lj}
   \tilde\LLj\phi(i,j)
   \defeq
   \lambda(v_{i})\,
   \Bigl(
     \phi\bigl(\bar i,(\bar j +j)\wedge N_{w} \bigr)
     -
     \phi\bigl(i,j\bigr)
   \Bigr)\,.
\end{align}
The associated forward Kolmogorov equation is:
\begin{align*}
  \dot \mu_{i',j'}(t) 
  &= 
  \sum_{i,j} \mu_{i,j}(t)\,\tilde\LLj_{(i,j),(i',j')}
\end{align*}
where $\tilde\LLj_{(i,j),(i',j')}=\tilde\LLj\phi(i,j)$ with $\phi(i,j)=\delta_{i i'}\,\delta_{j j'}$.
Straightforward calculations lead to the following system of ODEs:
\begin{equation}
\label{eq.kolmo2}
  \left\{
  \begin{array}{rll@{\hskip-0.5em}l}
  \dot \mu_{i',j'}(t) 
  &= 
  \displaystyle
  -\lambda(v_{i'})\,\mu_{i',j'}(t)
  && \hskip-2em\textrm{for }i'\neq \bar i\textrm{ and all }j'\,,
\\[0.7em]
  \dot \mu_{\bar i,j'}(t) 
  &= 
  \displaystyle
  -\lambda(v_{\bar i})\,\mu_{\bar i,j'}(t)
  &&  1\leq j'\leq \bar j\,,
\\[0.7em]
  \dot \mu_{\bar i,j'}(t) 
  &= 
  \displaystyle
  -\lambda(v_{\bar i})\,\mu_{\bar i,j'}(t)
  +\sum_{1\leq i\leq N_{v}}\lambda(v_{i})\,\mu_{i,j'-\bar j}(t)
  && \bar j<j'< N_{w}\,,
\\[0.7em]
  \dot \mu_{\bar i,N_{w}}(t) 
  &= 
  \displaystyle
  -\lambda(v_{\bar i})\,\mu_{\bar i,N_{w}}(t) 
  +
  \sum_{\substack{1\leq i\leq N_{v}\\ N_{w}-\bar j\leq j\leq N_{w}}}
      \lambda(v_{i})\, \mu_{i,j}(t)\,.
  \end{array}
  \right.
\end{equation}
Note that this last equation is:
\[
   \dot \mu_{\bar i,N_{w}}(t) 
  = 
  \sum_{\substack{N_{w}-\bar j\leq j\leq N_{w}\\ (i,j)\neq (\bar i,N_{w})}} 
      \lambda(v_{i})\, \mu_{i,j}(t)
\]
which indeed corresponds to the fact that $(\bar i,N_{w})$ is an accumulation point.

Hence, the update $\mu \to \mu'$ on interval $[0,\Delta t]$ consists of solving \eqref{eq.kolmo2} with initial condition $\mu$ and to set $\mu'=\mu(\Delta t)$.
System \eqref{eq.kolmo2} can be explicitly solved but not in a convenient way. As the previous step of the splitting method is of the first order, it is consistent to propose a simple first order approximation of \eqref{eq.kolmo2}.

Note that the solutions of the first two sets of equations in \eqref{eq.kolmo2} are  $\mu_{i',j'}(t)
  =
  e^{-\lambda(v_{i'})\,t}\,\mu_{i',j'}(0)$
for all $(i',j')$ not in $\{\bar i\}\times \{\bar j+1,\dots,N_{w}\}$. For the other components of \eqref{eq.kolmo2}, we choose to make an approximation which features the same time Euler scheme as \eqref{eq.kolmo2} and which respects both positivity and mass conservation properties, namely:
\begin{align}
\label{eq.split2}
   \mu' = \mB_{\Delta t}\, \mu
\end{align}
defined by:
\begin{equation}
\label{eq:B}
\left\{
\begin{array}{rlll}
  \mu'_{i',j'}
  &= 
  \displaystyle
  e^{-\lambda(v_{i'})\Delta t}\,\mu_{i',j'}\,,
  && \hskip-10em\textrm{for }i'\neq \bar i\textrm{ and all }j'\,,
\\[0.7em]
  \mu'_{\bar i',j'}
  &= 
  \displaystyle
  e^{-\lambda(v_{\bar i'})\Delta t}\,\mu_{\bar i',j'}\,,
  && \hskip-6em 1\leq j'\leq \bar j\,,
\\[0.7em]
  \mu'_{\bar i',j'}
  &= 
  \displaystyle
  e^{-\lambda(v_{\bar i})\Delta t}\,\mu_{\bar i,j'}
  +
  \sum_{1\leq i\leq N_{v}}
     (1-e^{-\lambda(v_{i})\Delta t})\,\mu_{i,j'-\bar j}\,,
  && \hskip-6em \bar j<j'< N_{w}\,,
\\[0.7em]
  \mu'_{\bar i,N_{w}}
  &= 
  \displaystyle
  e^{-\lambda(v_{\bar i})\Delta t}\mu_{\bar i,N_{w}} 
  + 
  \sum_{1\leq i\leq N_{v}}
     (1-e^{-\lambda(v_{i})\Delta t})\,\sum_{N_{w}-\bar j\leq j\leq N_{w}} \mu_{i,j}\,.
\end{array}
\right.
\end{equation}
Operator $\mB_{\Delta t}$ is first order accurate in time. Indeed, the first order of the Taylor expansion of $\mB_{\Delta t}\mu$ in $\Delta t$ corresponds to the Euler scheme for \cref{eq.kolmo2}.

\begin{remark}
	The case $\bar w<0$ can easily be adapted.
\end{remark}

\subsection*{Time step strategy}

For the iteration $\mu^n\rightarrow \mu^{n+1}$ we use a symmetric Strang splitting method of the second order \cite[p. 82]{glowinski2016a}:
\begin{align}
\label{eq.strang}
    \mu^{n+1} 
    = 
    \mB_{\Delta t/2}\,\mA_{\Delta t}(\mu^n)\,\mB_{\Delta t/2}\,  \mu^n
\end{align}
coupled with an adaptation of the classical Euler-Richardson extrapolation strategy for the control the time step size \cite{farago2009a,zlatev2017a} (see Algorithm \ref{algo.2}). 

\begin{algorithm}[H]
\begin{algorithmic}
  \STATE $\epsilon$ given tolerance parameter
  \STATE $\mu_{1/2} 
    \leftarrow 
    \mB_{\Delta t/4}\,\mA_{\Delta t/2}(\mu^n)\,\mB_{\Delta t/4} \, \mu^n$ 
  \STATE $\mu_1 
    \leftarrow 
    \mB_{\Delta t/4}\,\mA_{\Delta t/2}\,(\mu_{1/2})\mB_{\Delta t/4} \,\mu_{1/2}$ 
  \STATE  $e \leftarrow \Delta v \,\Delta w \,
            \| (\mu_{1/2} - \mu_1) \,v \|_{l^1(\Omegad)}$ (evaluation of the evolution of the solution)
  \IF{$e < \epsilon$}
     \STATE  $\mu^{n+1} \leftarrow \mu_1$
     \STATE  $t^{n+1} \leftarrow t^n + \Delta t$
  \ENDIF
  \STATE $\Delta t \leftarrow 0.9\,\sqrt{\epsilon/e} \,\Delta t$
\end{algorithmic}
\caption{Time step adaptation algorithm}
\label{algo.2}
\end{algorithm}

The choice of the indicator for the evaluation of the evolution of the solution has been made in order to accurately approximate the mean membrane potential $\crochet{\mu(t),\phi_{1}}$ with $\phi_{1}(v,w)=v$. It is indeed a good measure of error as the membrane potential of each individual neuron explodes in finite time. Another reasonable choice would be to control the mean firing rate $\crochet{\mu(t),\lambda}$. However, for problems with localised (in time) activity such as in Figure~\ref{fig:Hopf}~Bottom, the indicator would be small during small network spiking activity independently of the underlying dynamics, that is why we do not use this.

Note that in contrast to the classic Euler-Richardson algorithm, we do not  write:
  \[
  \mu^{n+1} = \mu^n + \Delta t \,\mD(\mu^n) \, \mu^n\,.
  \]
  This solution gives a second order accuracy for a fixed time step, but the explicit formulation would break the positivity property of our scheme.

\begin{remark}
	We did not prove the convergence of our approximate solution to the solution of the PDE \cref{macro} as $N_v,N_w\to\infty$. This requires at least to prove that the PDE \cref{macro} is well posed and to provide some properties regarding its dynamics (e.g. a priori bounds...). Nevertheless, we can readily see that our scheme above is incomplete. We need to set $v_{\max},w_{\max}$ and $v_{\min},w_{\min}$ such that our approximation of the original jump process by the absorbing one does not affect the dynamics too much. The accumulating point $(\bar i,w_{\max})$ (resp. $(\bar i,w_{\min})$) in the case $\bar w>0$ (resp. $\bar w<0$) should be compensated by a redistribution of the mass by the drift if one wants to avoid concentration of mass that is not a feature of PDE \cref{macro}. A simple way to achieve this is, for example, to chose $w_{\max}$ high enough so that $(\bar i,w_{\max})$ is above the w-nullcline and the v-nullcline. This way, the flow is downward to the left at $(\bar i,w_{\max})$ and any mass at this point will be re-injected into the dynamics. Then, we chose $v_{min}$ small enough, on the left of the $v$-nullcline, to ensure that the vector field is entrant in the domain. Finally, we chose $v_{max}$ large enough, on the right of the $v$-nullcline to capture the explosive behaviour.
\end{remark}

\subsection{Properties}

Define the mass of a discrete solution $\mu$ as
\[
  m(\mu) \defeq \sum_{j=1}^{N_w}\sum_{i=1}^{N_{v}} \mu_{i,j}\,.
\]
We have designed the numerical schemes so that $\mA_{\Delta t}(\mu^n)$ and $\mB_{\Delta t}$ are mass conservative. It follows that the general Algorithm~\ref{algo.2} is mass conservative as well. Let us now focus on proving the positivity of the algorithm. Note that the same property was proved for a general mesh in \cite{boyer_analysis_2012} using a different method. It should be straightforward to adapt our proof to a general (non-regular) mesh. The main idea of our proof is to note that the adjoint of the discrete divergence $\mD^*$ is diagonally dominant.

\begin{theorem}
\label{theo}
For $\Delta t>0$ and $\mu^0\geq 0$, we have the following properties:
	\begin{enumerate}
		\item there exists a unique solution $\nu$ to the equation:
		\[
		\mA_{\Delta t}(\mu^0)\, \nu = \mu
		\]
		with null fluxes on the boundaries \cref{BCnumerics},		
		\item if $\mu$ is non-negative, then so is $\nu$.
	\end{enumerate}
	
\end{theorem}
\begin{proof}
	
For simplicity, as we work at time fixed in this proof, we drop the dependency of $\mA$ and $\mD$ on $\mu^0$.
The idea of the proof is to show that $I+\Delta t\,\mD^*$ is diagonally strictly dominant hence invertible with positive diagonal and non-negative off-diagonal elements. To this end, we first have to identify the adjoint $\mD^*$ of $\mD$. For this, we consider the following quantity
\begin{align*}
 \crochet{\mu,\mD^*\phi}
 &=
 \crochet{\mD\mu,\phi}
 =
 \sum_{i,j}
    \phi_{i,j}\,
    \bigl[(F_{i+1/2,j}-F_{i-1/2,j})+(G_{i,j+1/2}-G_{i,j-1/2})\bigr].
\end{align*}
Using the boundary conditions \cref{BCnumerics}, we get:
\begin{equation*}
 \crochet{\mu,\mD^*\phi}
 =
  \sum_{j=1}^{N_w}\sum_{i=1}^{N_v-1}\left(\phi_{i,j}-\phi_{i+1,j}\right)F_{i+1/2,j}
  +
  \sum_{i=1}^{N_v}\sum_{j=1}^{N_w-1}\left(\phi_{i,j}-\phi_{i,j+1}\right)G_{i,j+1/2}\,.
\end{equation*}
Note that from \cref{flux}, for any $\psi\in \mathbb R^{N_v-1}$:
\begin{align*}
  &\sum_{i=1}^{N_v-1}\psi_i F_{i+1/2,j}
  =
  \sum_{i=1}^{N_v-1}
    \psi_i\,\bigl(V_{i+1/2,j}^+ \, \mu_{i,j} 
                     - V_{i+1/2,j}^- \, \mu_{i+1,j}\bigr)
\\
  &\qquad 
  =
  \mu_{1,j}\,V^+_{3/2,j}\,\psi_1
  -
  \mu_{N_v}\,V^-_{N_v-1/2,j}\,\psi_{N_v-1}
  +
  \sum_{i=2}^{N_v-1}
     \mu_{i,j} \,\bigl(V_{i+1/2,j}^+\,\psi_i-V_{i-1/2,j}^-\,\psi_{i-1}\bigr)
\end{align*}
hence:
\begin{multline}
\label{eq:adjoint_i}
  \sum_{j=1}^{N_w}\sum_{i=1}^{N_v-1}\left(\phi_{i,j}-\phi_{i+1,j}\right)\,F_{i+1/2,j}
\\
  =
  \sum_{j=1}^{N_w}
  \Biggl\{
  \sum_{i=2}^{N_v-1}
    \mu_{i,j}\,
    \Bigl[
       (V_{i+1/2,j}^++V_{i-1/2,j}^-)\,\phi_{i,j}
       - V_{i+1/2,j}^+\,\phi_{i+1,j}
       - V_{i-1/2,j}^-\,\phi_{i-1,j}
    \Bigr]
    \hspace{3em}
\\[-0.7em]
  +
  \mu_{1,j}\,V^+_{3/2,j}\,(\phi_{1,j}-\phi_{2,j})
  -
  \mu_{N_v,j}\,V^-_{N_v-1,j}\,(\phi_{N_v-1/2,j}-\phi_{N_v,j}) 
  \Biggr\}
  \,.
\end{multline}
Similarly:
\begin{multline}
\label{eq:adjoint_j}
  \sum_{i=1}^{N_v}
  \sum_{j=1}^{N_w-1}
  (\phi_{i,j}-\phi_{i,j+1})\,G_{i,j+1/2}
\\
  =
  \sum_{i=1}^{N_v}
  \Biggl\{
  \sum_{j=2}^{N_w-1}\mu_{i,j}
    \Bigl[
      (W_{i,j+1/2}^++W_{i,j-1/2}^-)\,\phi_{i,j}
       - W_{i,j+1/2}^+\,\phi_{i,j+1}
       - W_{i,j-1/2}^-\,\phi_{i,j-1}
    \Bigr]
    \hspace{2em}
\\[-0.7em]
  +
  \mu_{i,1}\, W^+_{i,3/2} \, (\phi_{i,1}-\phi_{i,2})
  -
  \mu_{i,N_w} \, W^-_{i,N_v-1} \, (\phi_{i,N_w-1/2}-\phi_{i,N_w}).
  \Biggr\}
\end{multline}
From \cref{eq:adjoint_i} and \cref{eq:adjoint_j}, the diagonal element of $\mD^*$ are non-negative, the off-diagonal elements are non-positive. Finally, the matrix $\mD^*$ is  diagonally dominant (not stric\-tly). Indeed, each term in  \cref{eq:adjoint_i}--\cref{eq:adjoint_j} satisfies this property and the set of matrices satisfying this property is obviously convex. 

Let us now conclude.
First $M=I+\Delta t\,\mD^*$ is diagonally strictly dominant hence invertible.
Second $M$ has a non-negative (resp. non-positive) diagonal (resp. off-diagonal), so up to a scaling, 
we can assume that $\max_{i}M_{ii}<1$, so that $P\defeq I-M$ is non-negative and $\norm{P}_{\infty}<1$. We can thus expand $M^{-1}=(I-P)^{-1}$ as an infinite sum $I+P+P^2+\cdots$ which is non-negative too. The proof is complete.
\end{proof}

\medskip

\begin{lemma}
  $\mB_{\Delta t}$ is non-negative in the sense that for all non-negative $\mu$, $\mB_{\Delta t} \, \mu$ is also non-negative.
\end{lemma}

\begin{proof}
  From the expression \cref{eq:B} of $\mB_{\Delta t}$, this is a consequence of $\lambda\, \Delta t$ being non-negative.
\end{proof}

\medskip

\begin{proposition}
  The full numerical scheme based on the splitting strategy consisting of formulas \cref{FVM} and \cref{eq:B} is non-negative. This is the discrete version of the fact that the solution of \cref{eq.transport.weak} is a positive measure.
\end{proposition}

\begin{proof}
This result is a consequence of the previous results on non-negativity of $\mA_{\Delta t}(\mu)$ and $\mB_{\Delta t}$.
\end{proof}

\section{Numerical simulations}
\label{sec.simulations}

Let us define useful statistical quantifiers in order to analyse the results,
mean membrane potential for both descriptions:
\[
  V_N = \frac{1}{N}\sum_i v_i\,, 
  \quad 
  V_{\infty} = \sum_{i,j}v_i\,\mu_{i,j}.
\]

The implementation is done in Julia language. For the Algorithm~\ref{algo.2}, the linear system associated to $\mA_{\Delta t}(\mu^n)$ is encoded in a sparse matrix at each iteration and a general linear solver is then called.
GPU simulations of the finite size network were carried on a Nvidia Tesla K80 card. In all simulations, we used a Mersenne twister for the generation of random numbers.

\begin{table}[h!]
\begin{center}
  \begin{tabular}{ | c | c | c | c |}
    \hline
    Parameter/function & CV test & Invariant distributions & Hopf test \\
    \hline
    $I(t)$ & $2$ & $-2,1$ & $0$\\
    \hline
    $\tau_w$ & $1$ & $2$ & $13$\\
    $b$ &$1$  & $1,0.05$ & $0.011$\\
    $\bar{v}$ & $1.0$ & $1.8$ & $-1.5$\\
    ${\bar w}$ & $1.5$ & $5.5, 1.5$ & $1.5$\\
    \hline
    $J$ & $3.1$ & $0$ & $5$\\
    \hline
    $\lambda(v)$ & $0.1 + e^{(v-1)}$ & $e^v$ & $e^v$\\
    $F(v)$ & $e^v - 5v$ & $e^v-v$ & $e^v-v$\\
    \hline
  \end{tabular}
  \caption{Model parameters}
  \label{modelParam}
\end{center}
\end{table}

\subsection{Convergence and propagation of chaos}

In this paragraph, we study the behaviour of the finite size network \cref{micro} as the number $N$ of neurons is increased.
We use the default parameters (see \cref{modelParam}) and we start from a Gaussian distribution:
\begin{equation}
  \mu_0(\rmd v,\rmd w) 
  = 
  \frac{1}{2 \,\pi\, \sigma_1\, \sigma_2} 
  \exp\Bigl(
     -\frac{(v-\mu_1)^2}{2\,\sigma_1} - \frac{(w-\mu_2)^2}{2\,\sigma_2}
  \Bigr)
  \,\rmd v\,\rmd w
  \label{IC}
\end{equation}
with $\sigma_1 =\sigma_2 = 1$, $\mu_1 = -1.3$, $\mu_2 = 2.28$.

\begin{figure}[htbp!]
  \begin{center}
    \includegraphics[width=0.49\textwidth]{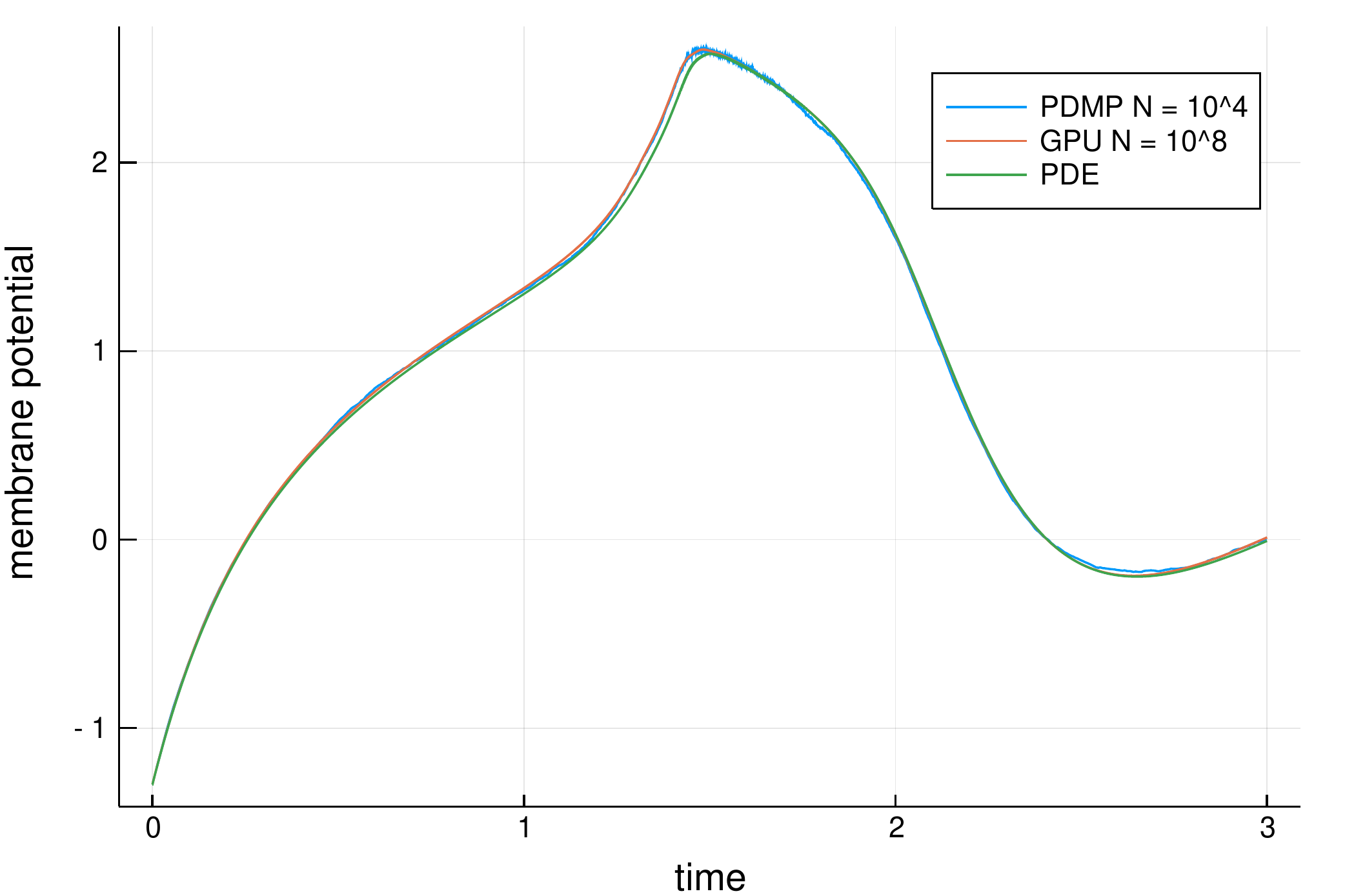}
    \\
    \includegraphics[width=0.49\textwidth]{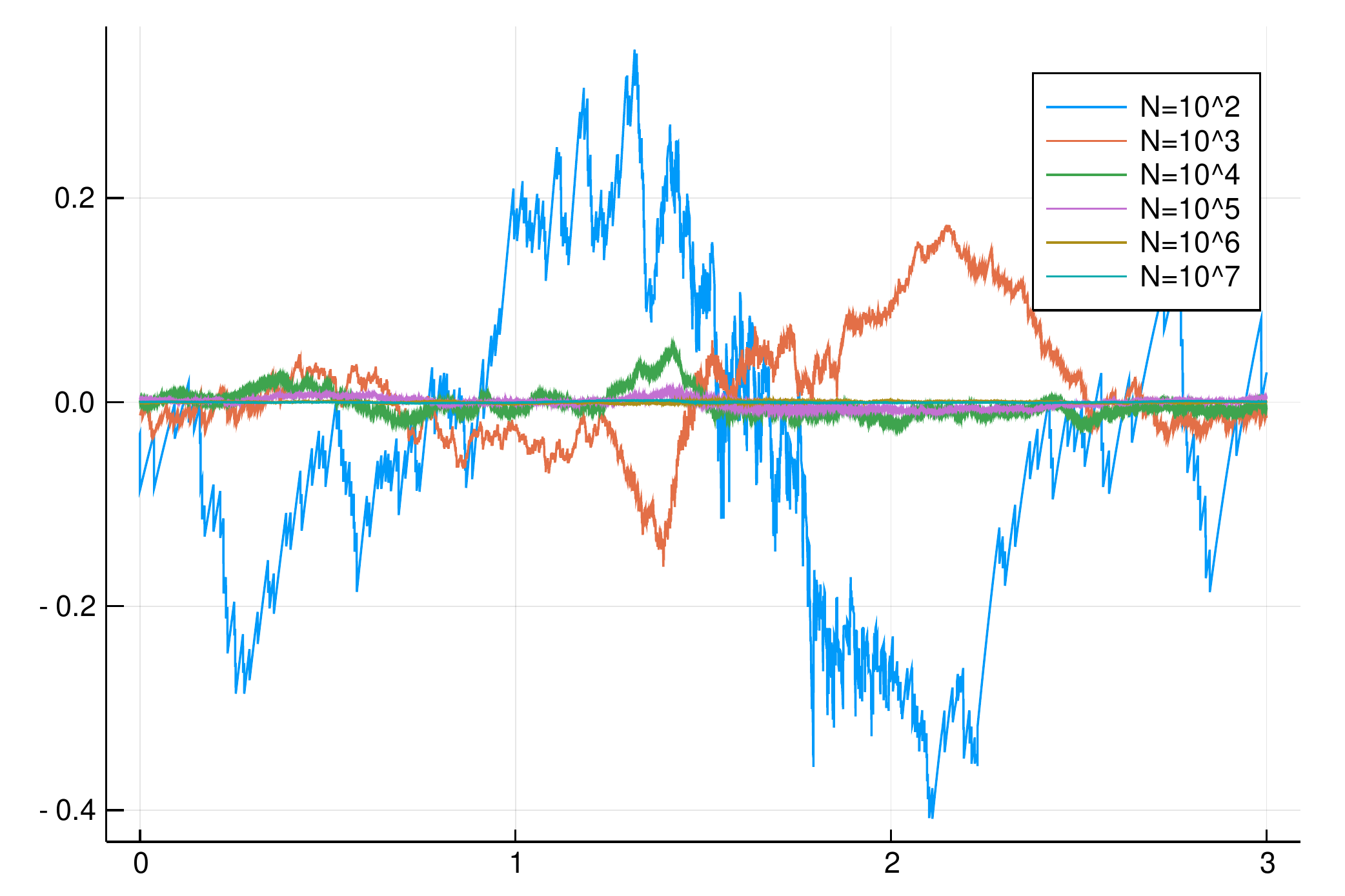}
    \includegraphics[width=0.49\textwidth]{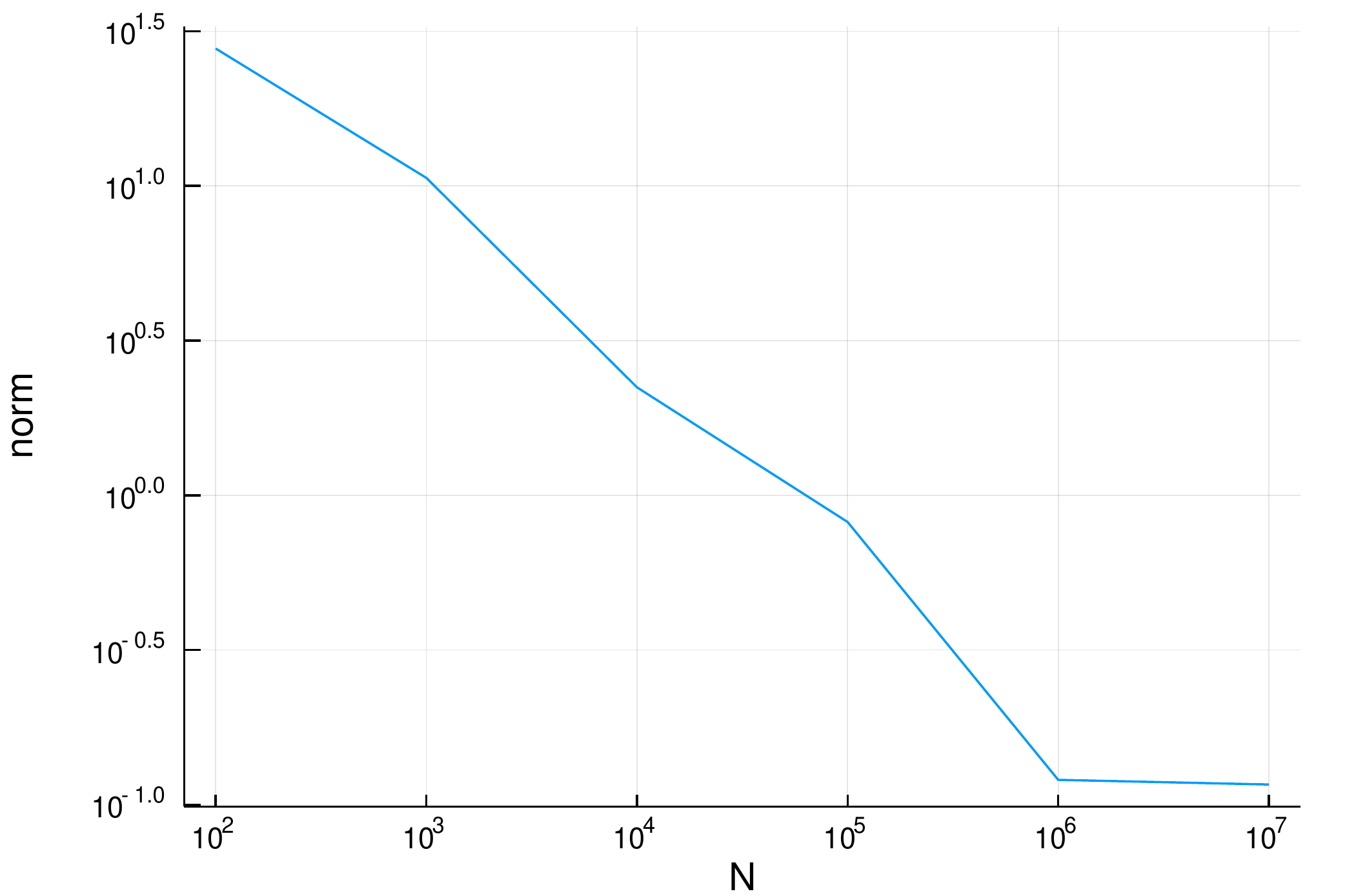}
    \\
    \includegraphics[width=0.49\textwidth]{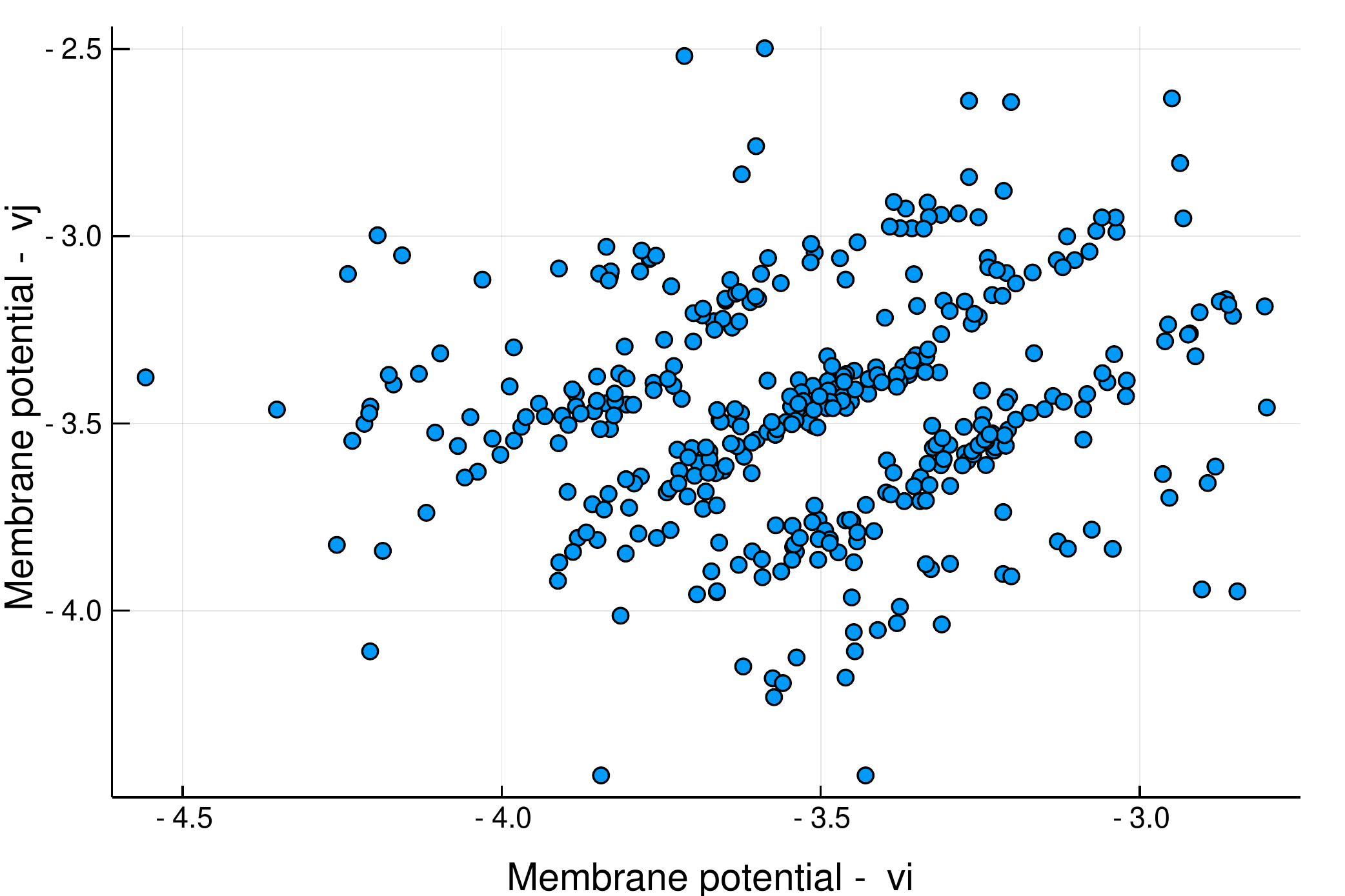}
    \includegraphics[width=0.49\textwidth]{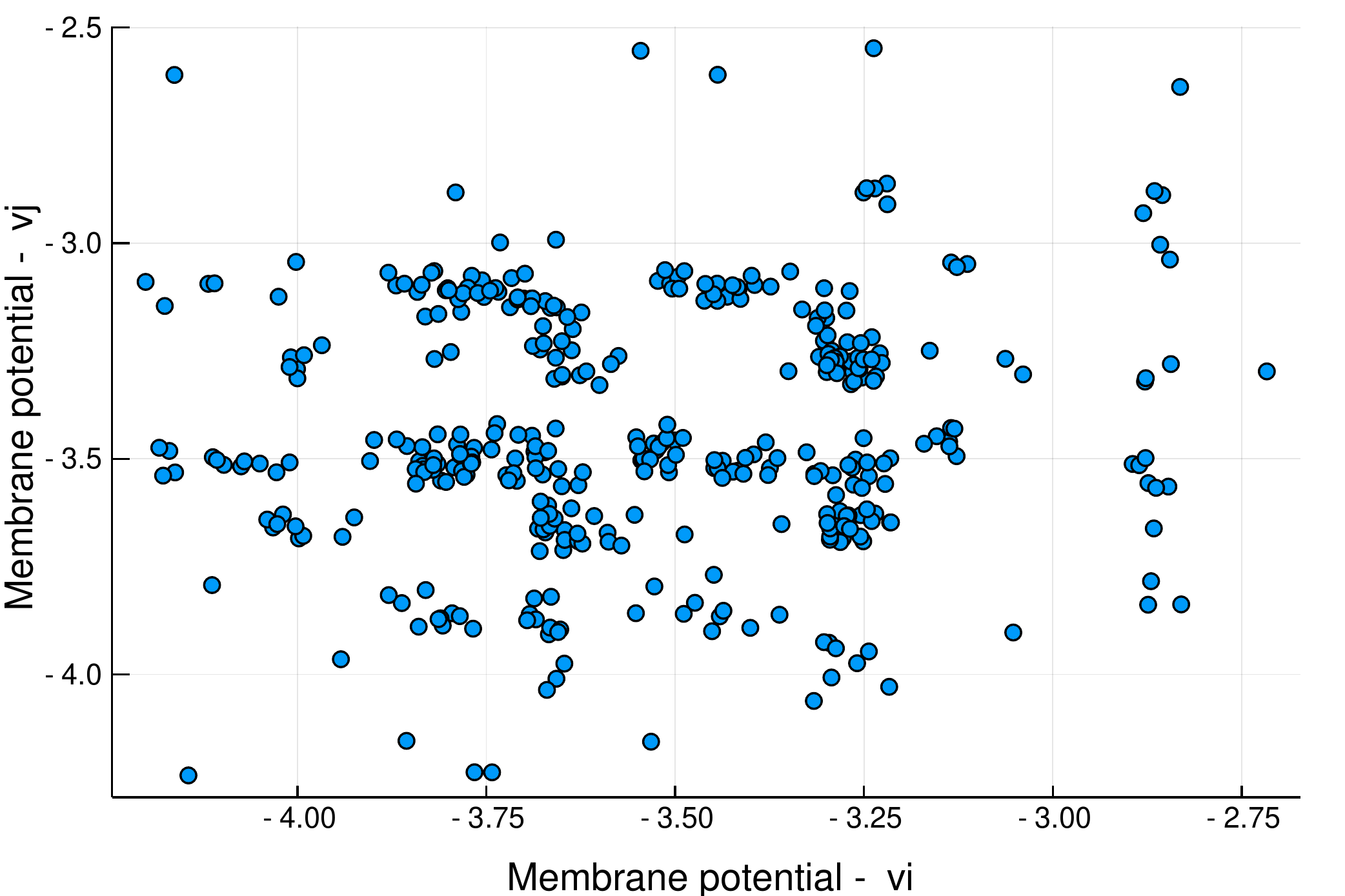}
  \end{center}
  \caption{Top: convergence test, comparison of the approximation obtained using three different methods (PDMP$^2$, PDMP$^1$ and PDE).
    Middle: difference of mean membrane potential as function of time and error as function of $N$ (reference solution $N=10^8$). Computed using the  Monte Carlo method.
    Bottom: quantification of propagation of chaos,
    particle correlation as a function of $N$,
    with each realisation as a simulation of \cref{micro} until a fixed final time corresponding to a stationary distribution.
  }
  \label{CVtest}
\end{figure}

On \cref{CVtest} (top) we present the results of a convergence test.
We plot the means membrane potential $V_N,V_\infty$ computed using two distinct approaches, microscopic and macroscopic.
For the microscopic approach, we have used two different methods. 
The label PDMP$^1$ refers to the method developed in this article
whereas the label PDMP$^2$ refers to the Julia library \href{https://github.com/rveltz/PiecewiseDeterministicMarkovProcesses.jl}{PiecewiseDeterministicMarkovProcesses.jl}, based on a stiff solver for the flow and a specific algorithm. This second method, based on \cite{veltz_new_2015}, is almost exact in that the errors come from the numerical flow in between jumps and the random generator. However, PDMP$^2$ cannot be used for a network size larger than $\sim10^4$ and that is why we use it as a reference for PDMP$^1$.
The three curves are in very good agreement.
We note that the convergence is linear in $N$ towards the mean-field approximation, and that a plateau is reached for $N=10^6$ (machine zero is reached).

We say that a $N$-neurons system, such as \cref{micro}, propagates chaos if the particles that compose the system become independent as the total number tends to infinity.
This concept was originally introduced in \cite{Kac1956} in the framework of kinetic theory.
In order to numerically support this property, we adapt the algorithm developed in \cite{Carlen2011}.
For each size of network $N$, we consider $M$ realisations.
Each realisation consists of simulating the system \cref{micro} until a final time $T$, long enough such that the system reaches a steady state.
Then, we randomly pick (uniform distribution) two neurons, let's say neuron $i$ and $j$, and we keep them.
At the end, we plot the 2D histogram in order to analyse the correlations.
On \cref{CVtest}  (bottom) we present the results of this test, with $N = 10^4:10^6$ and $M=400$.
A correlation is clearly visible for the case $N=10^4$, and seems to vanish as $N$ grows, as the cases $N=10^5$ seem to indicate.
This would be a numerical argument in favour of propagation of chaos in the case of this model.

\subsection{Equilibrium and bifurcation}

The invariant distributions $\mu^{\textrm{\tiny inv}}$ of \eqref{macro} are invariant distribution(s) of an isolated neuron, \textit{e.g.} for $J=0$, but for the current $I+J\crochet{\mu^{\textrm{\tiny inv}},\lambda}$. We thus have to look for the invariant distributions of an isolated neuron in order to study their existence in the nonlinear case. Two examples are shown in Figure~\ref{fig:inv-dist}~Left, in the case of two and no  equilibria for the underlying vector field. These results are obtained by using the Algorithm~\ref{algo.2} for very long times. It appears that the implementation respects the properties of positivity and conservation of mass. In the first case, one equilibrium is an attracting focus to which the dynamics are attracted. The isolated neuron is a PDMP with embedded jump chain $(\bar v,w_n)_n$ being a Markov chain. In the Right column, we compare the invariant distribution of the embedded chain with the (renormalised) quantity $\int_{\R}\lambda(v) \,\sigma_{\bar w} \,\mu^{\textrm{\tiny inv}}(\rmd v,\cdot)$. In the second case, the density presents several peaks on the reset line below the $v$ nullcline. This represents several consecutive spikes before a quiescent period during which the dynamics re-accumulate under the $v$-nullcline: this is typically a \textit{bursting} behaviour.
 
Based on extensive simulations, we make the following conjecture in the case $J=0$. 

\textit{The isolated neuron with rate function $\lambda=\exp$ and $F(v) = e^v-v$ has a unique invariant distribution to which the (linear) dynamics converge exponentially fast. Furthermore, this invariant distribution has a density with respect to the Lebesgue measure.
}

It seems that the conjecture should hold for $\lambda$ positive increasing and $F$ such that $\lambda/F$ is not integrable\footnote{to ensure that the network spikes at least once} at $v=+\infty$, but we don't have enough data to support this.

\begin{figure}[ht!]
		\includegraphics[width=0.47\textwidth]{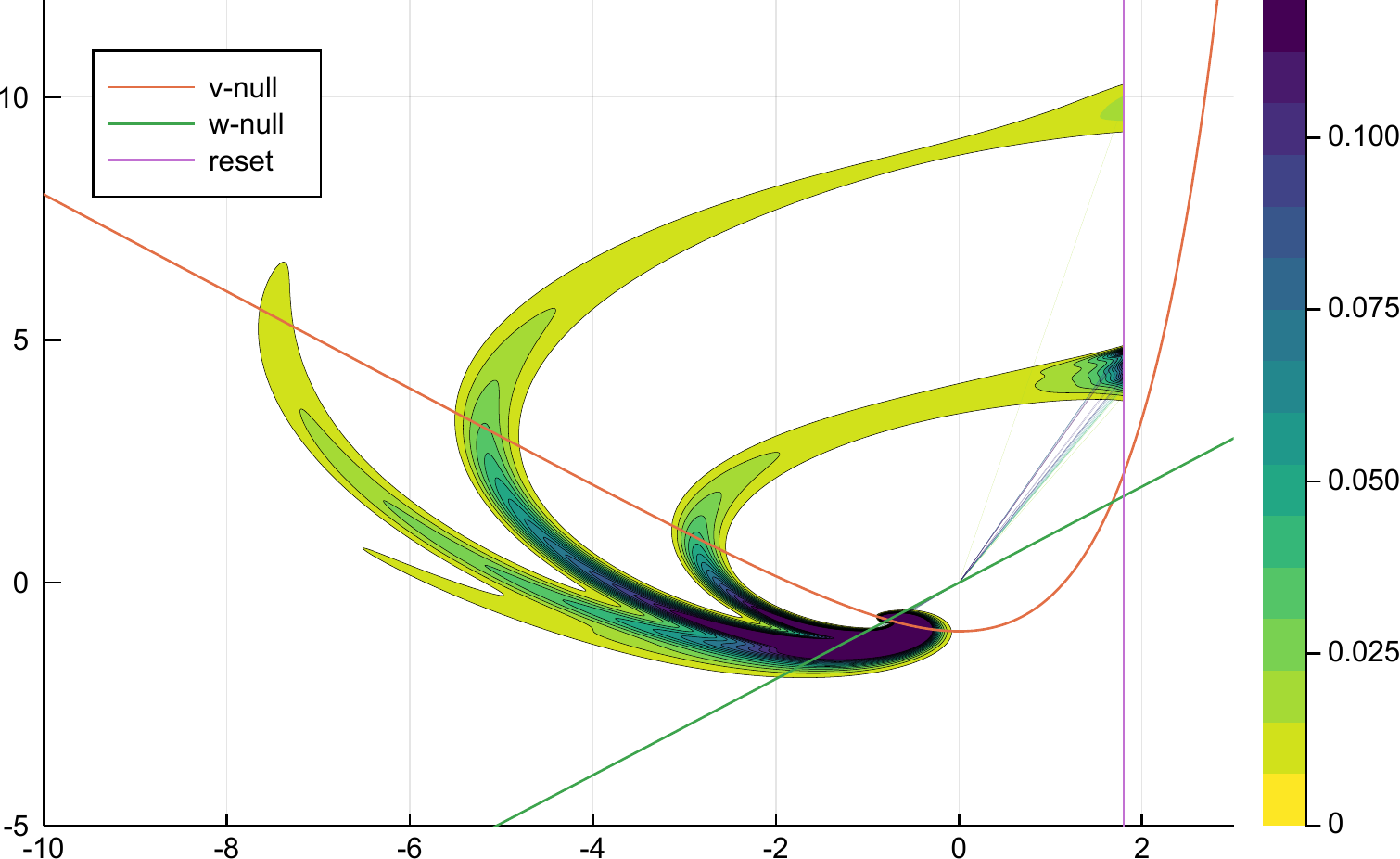}
	\includegraphics[width=0.47\textwidth]{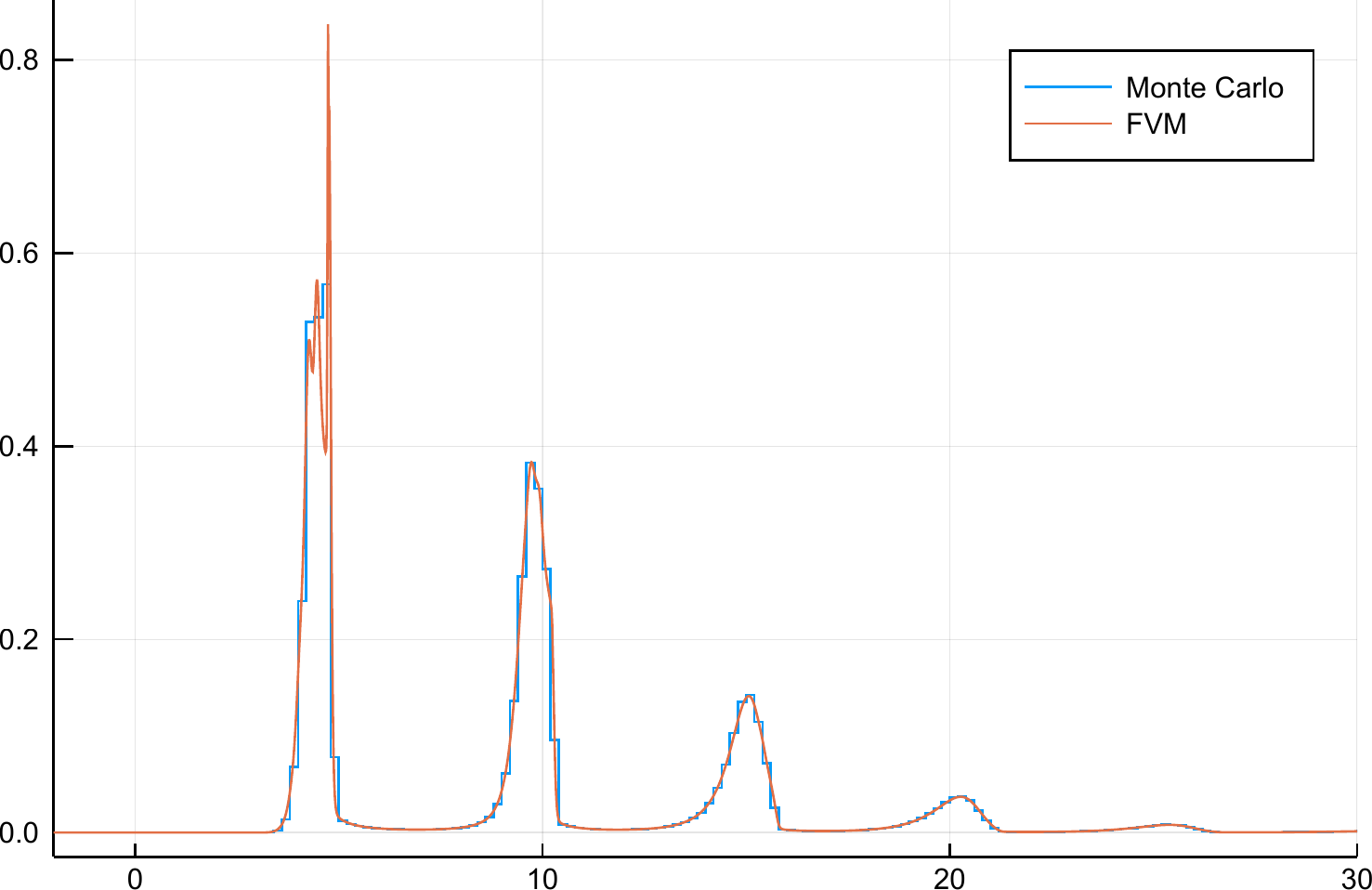}
		\includegraphics[width=0.47\textwidth]{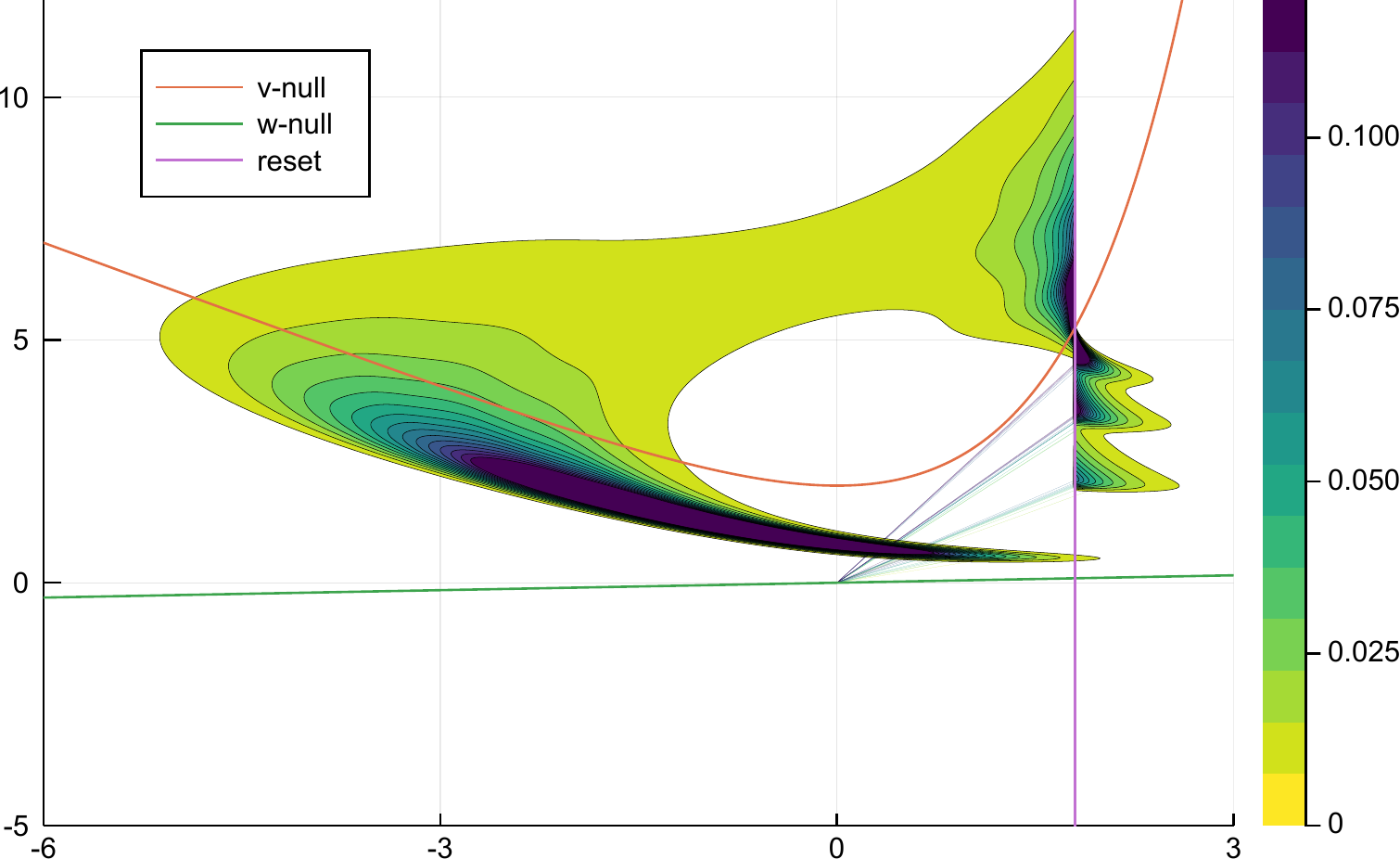}
		\includegraphics[width=0.47\textwidth]{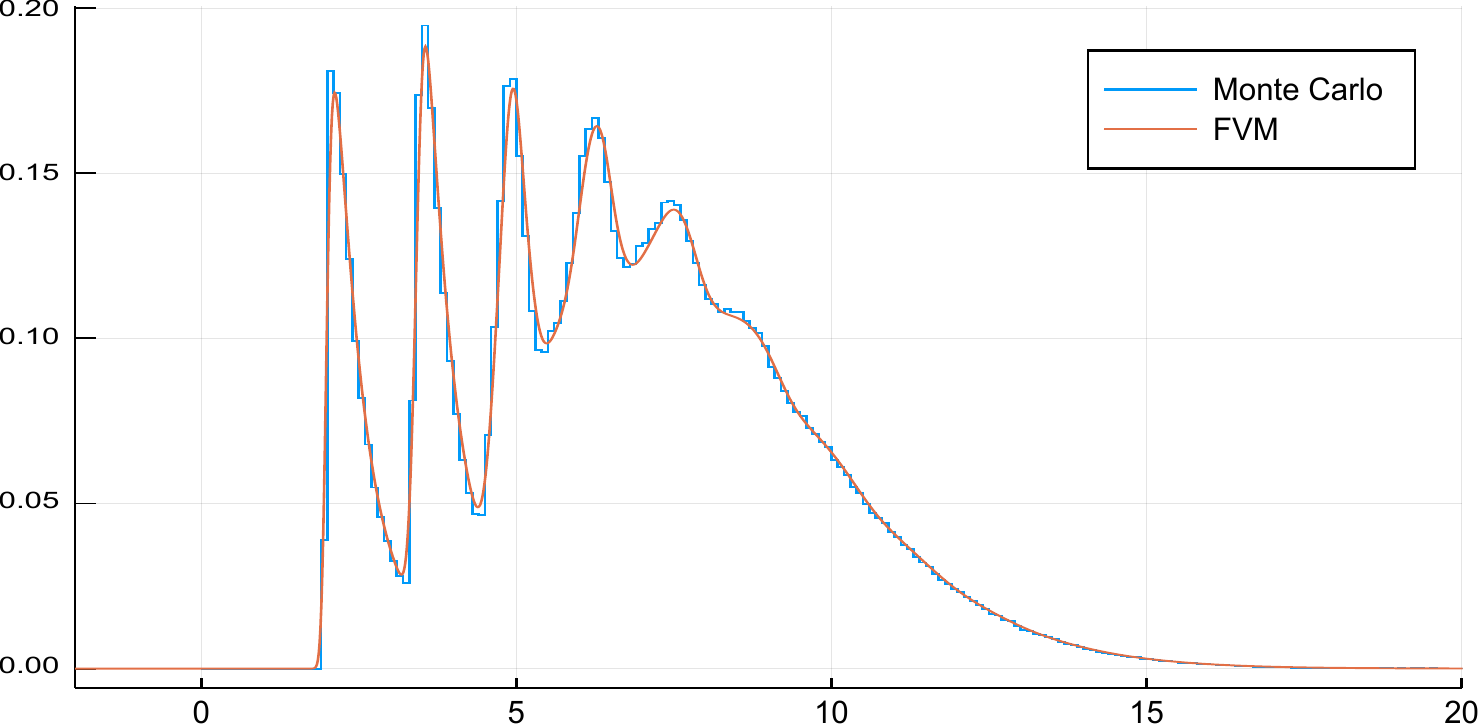}
	\caption{Top: First parameters in Table~\ref{modelParam}, column ``Invariant distributions''. Bottom: Idem but for the second parameters. Left: contour plot of the invariant distributions in the case of the isolated neuron ($J=0$). Colour axis are clipped on the contour plots in order to show structures that would be otherwise invisible. Right: comparison of the $w$-marginals with finite size system. We used $N_v=N_w=3000$ for the FVM. For the Monte Carlo, we simulated $2\cdot 10^6$ jumps.}
	\label{fig:inv-dist}
\end{figure}

The existence of invariant distributions for the nonlinear equation \eqref{macro} is much less trivial to study numerically. Based on the above conjecture, it should hold that there is a unique attracting invariant distribution when the connectivity $J$ is small. 

We therefore seek to study the effect of connectivity strength by varying the coupling parameter $J$.
Following the approach done in \cite{Drogoul2017}, we look for synchronised activity within the network.
We use same initial condition and parameters as in the previous paragraph, except for the parameter $J$.

Results are displayed on \cref{fig:Hopf}.
For values of $J$ smaller than $6.15$ (inset of~\cref{fig:Hopf}~Top~Left), the network is not synchronised, and tends to relax to an invariant distribution.
Above a threshold value $J^* \approx 6.15$, we observe synchronised activity within the network characterised by a periodic solution to \eqref{macro}; see \cref{cycle} for an example of a periodic orbit far from the ``Hopf bifurcation point''.
Figure~\ref{fig:Hopf} is a numerical evidence of the possible existence of a Hopf bifurcation for the network, based on the parameter $J$. Please note that we could not get the scaling behaviour \cite{kuznetsov2004} in $\propto\sqrt{J-J^*}$ for the amplitude of the periodic orbit (inset of~\cref{fig:Hopf}~Top~Left) possibly due to a subcritical bifurcation.

\begin{figure}[ht!]
  \begin{center}
    \includegraphics[width=0.4\textwidth]{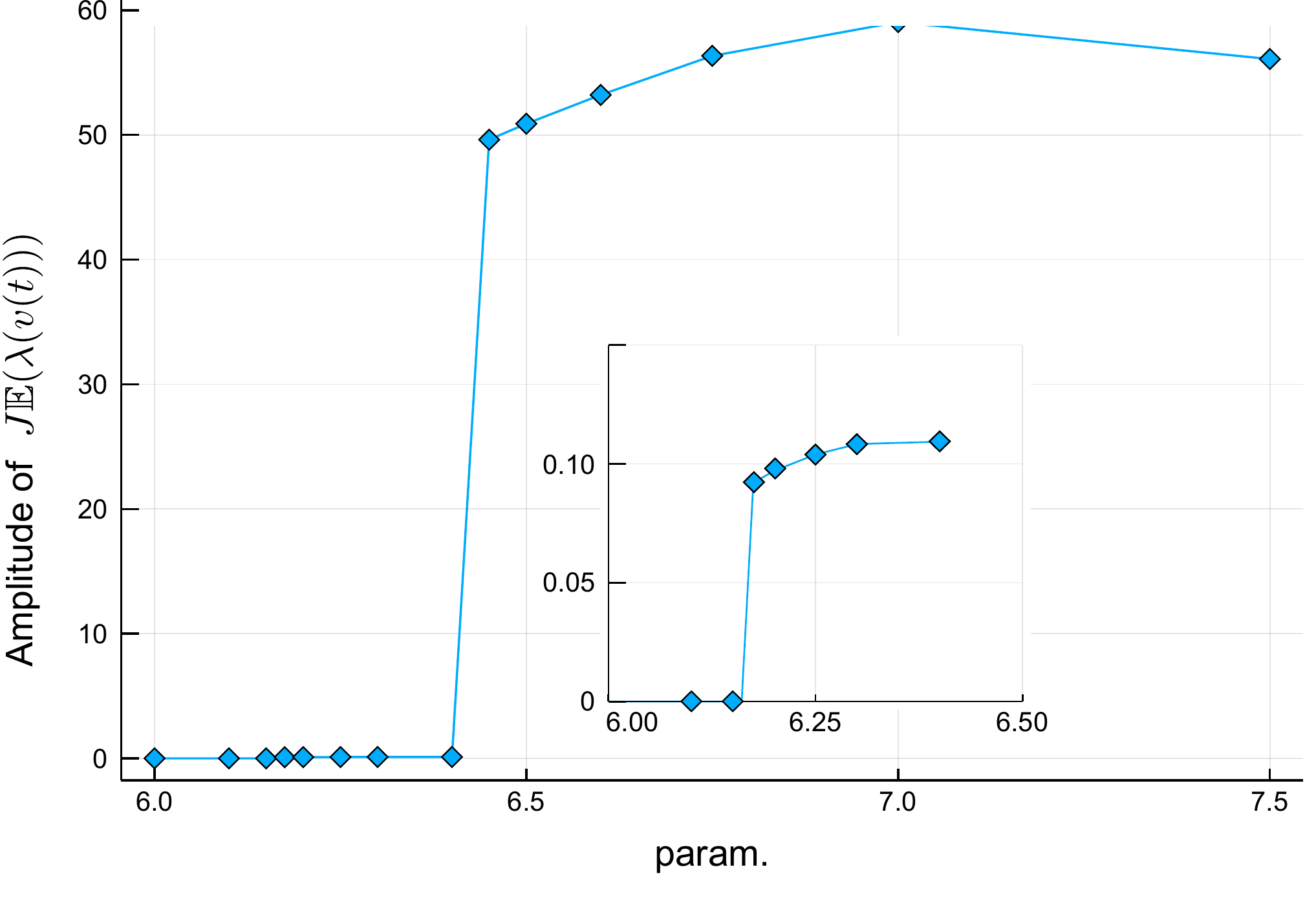}
    \includegraphics[width=0.45\textwidth]{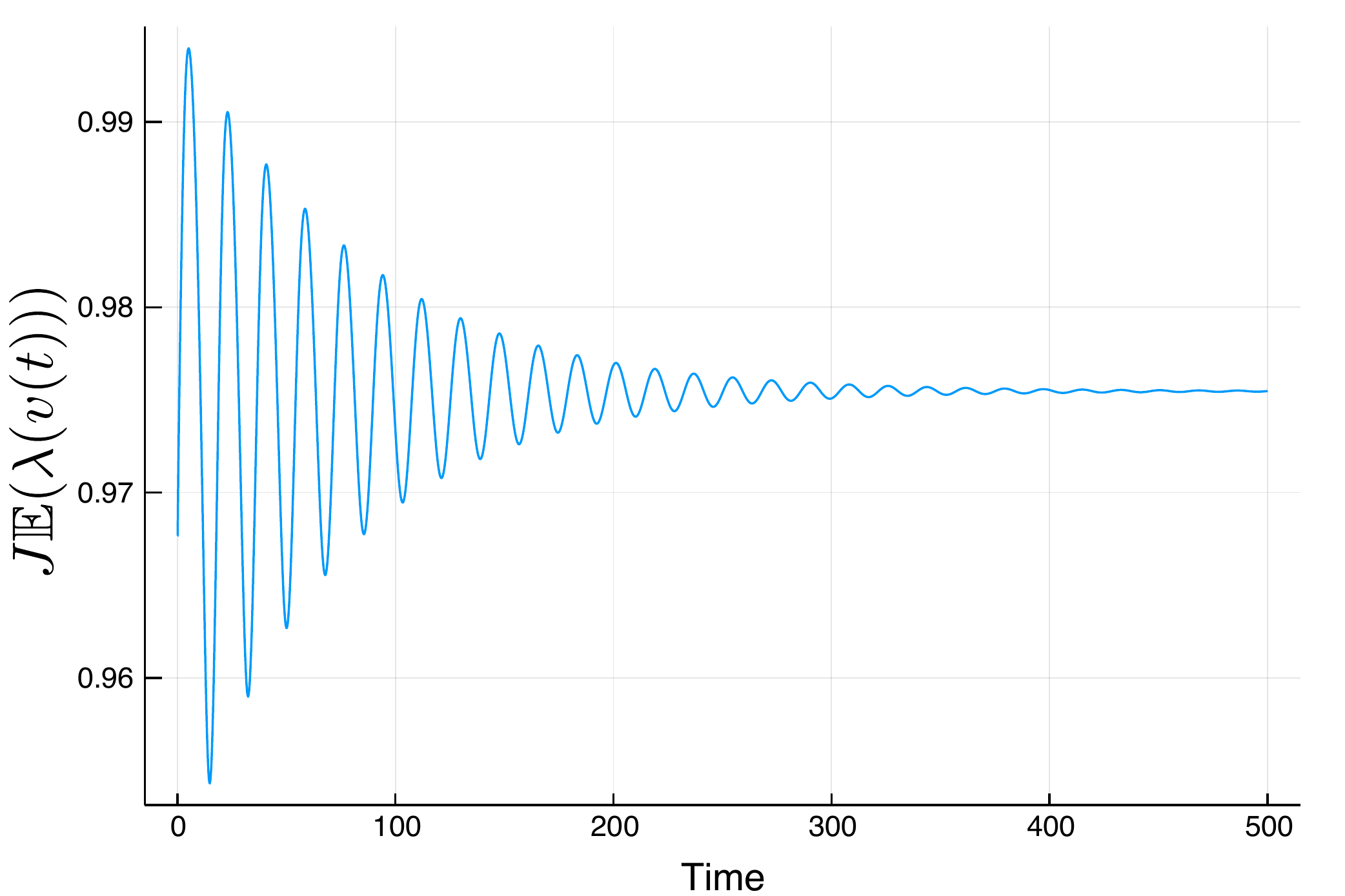}
    \includegraphics[width=0.45\textwidth]{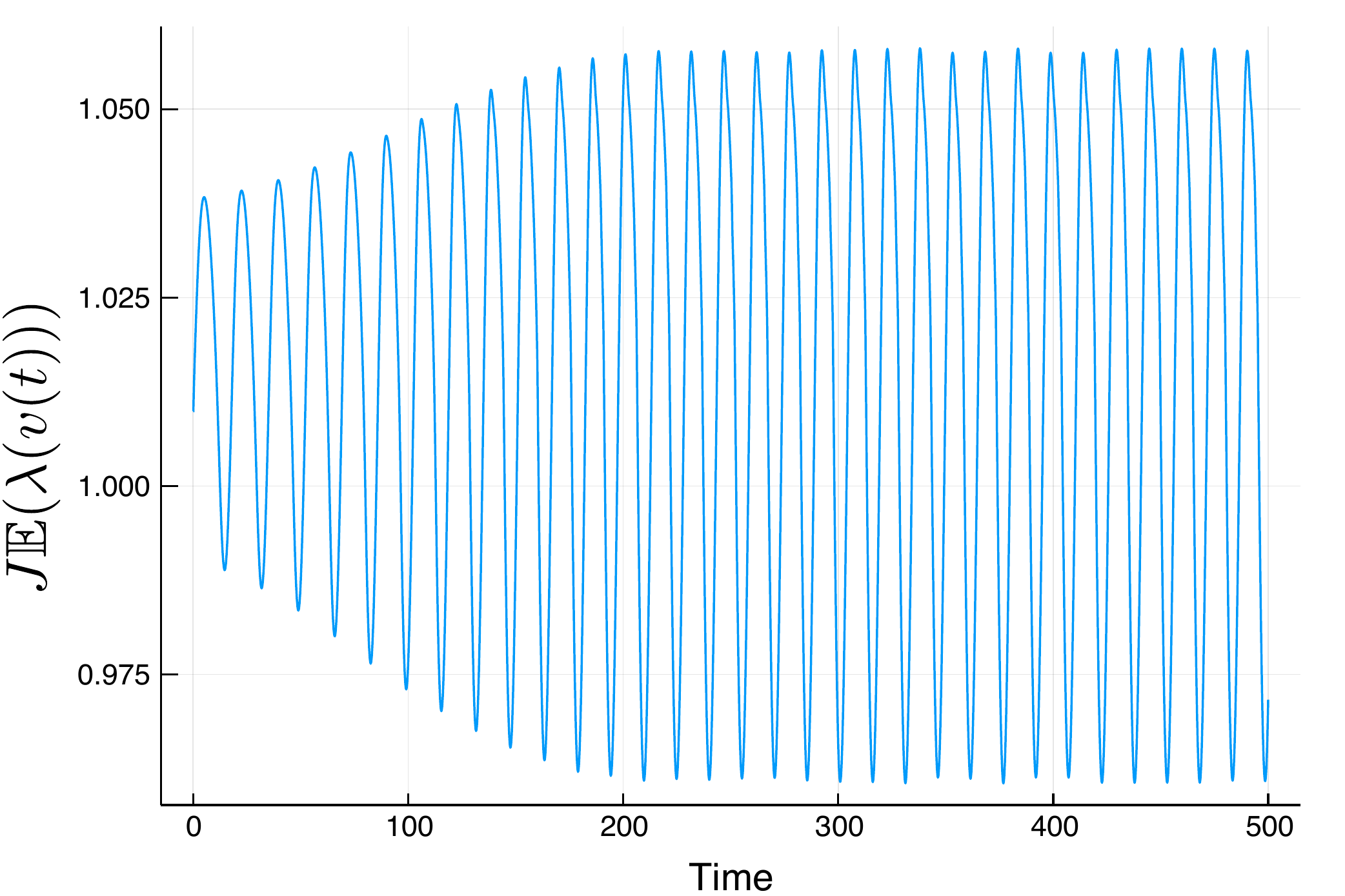}
     \includegraphics[width=0.45\textwidth]{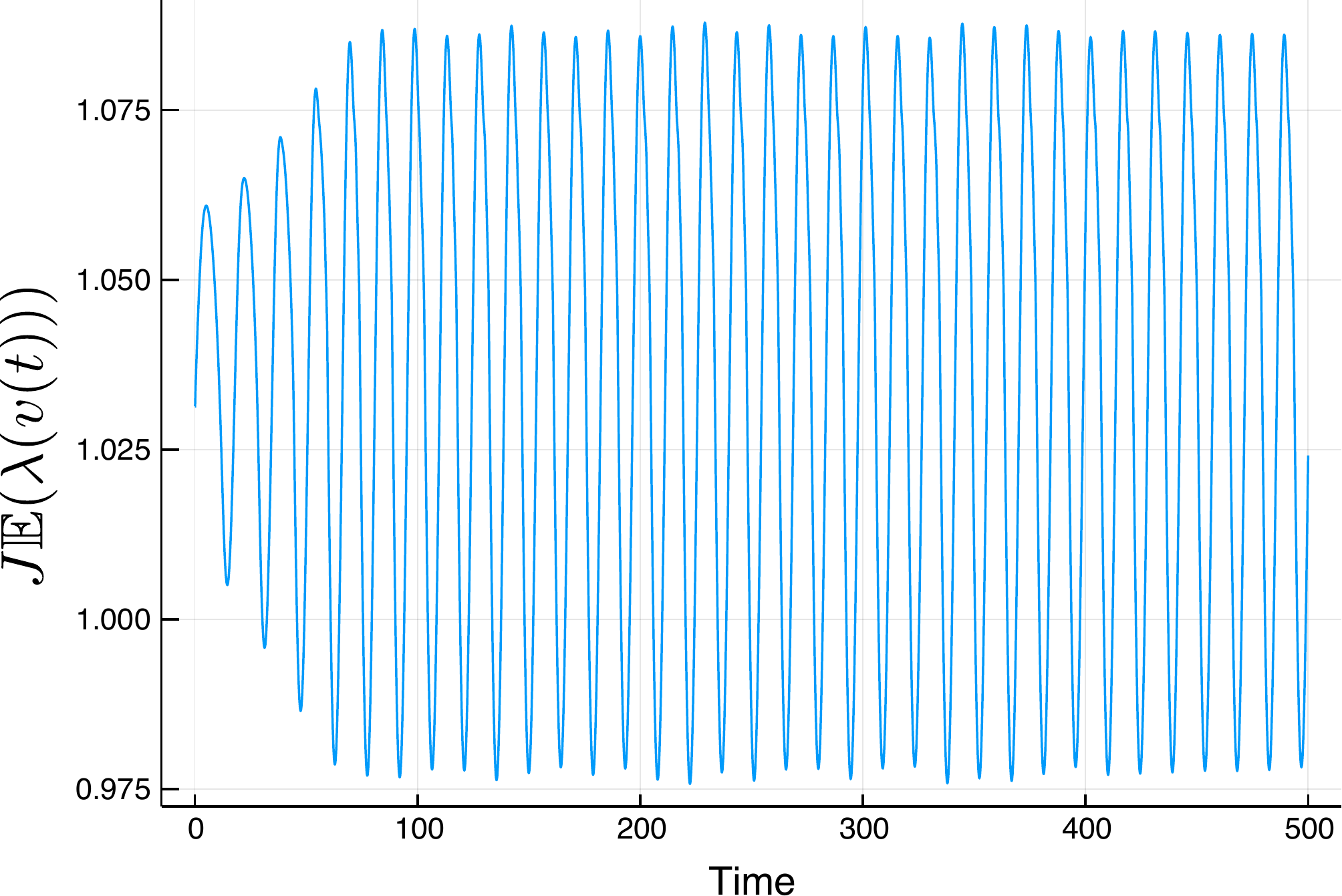}
    \includegraphics[width=0.45\textwidth]{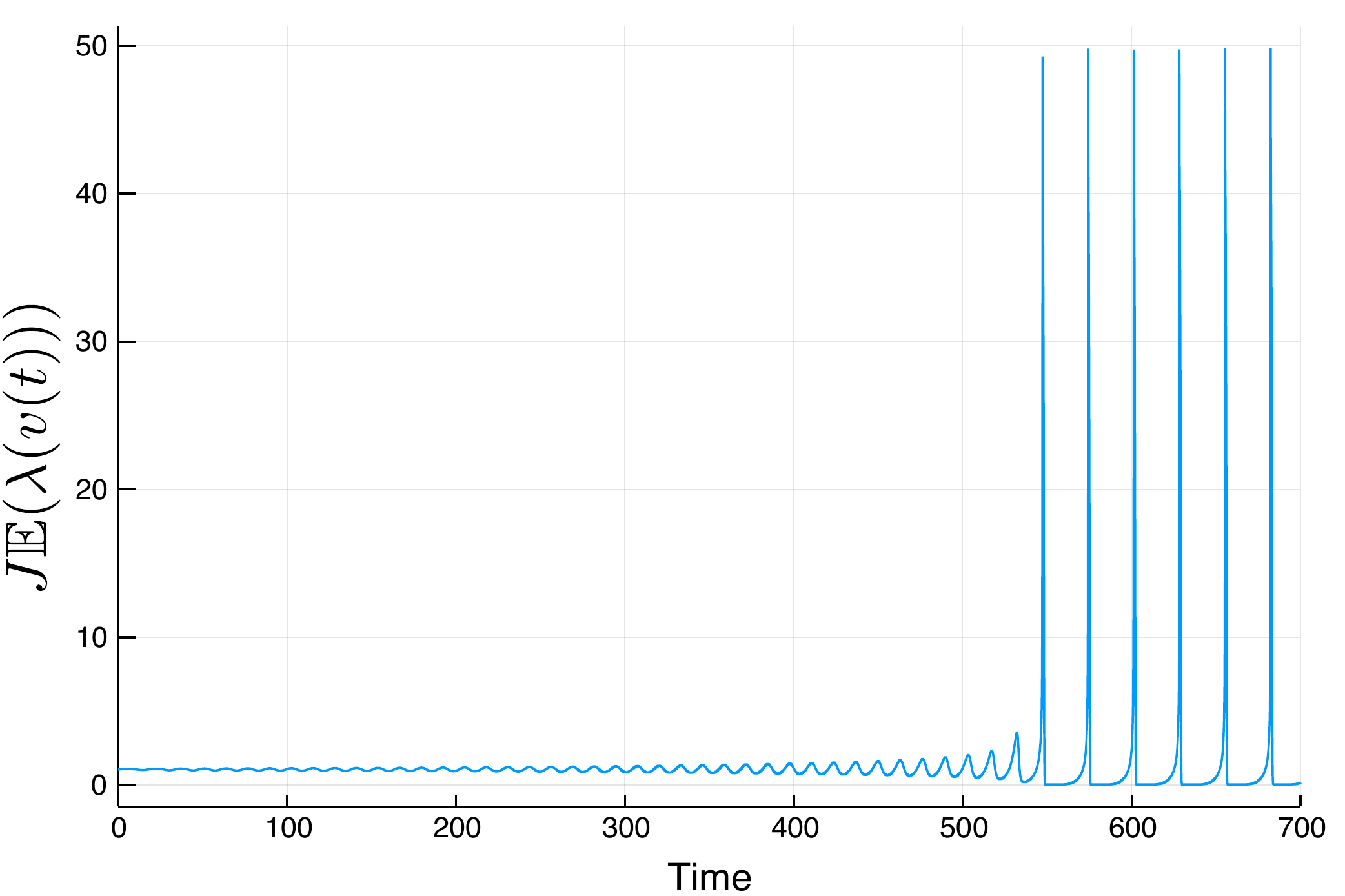}
     \includegraphics[width=0.45\textwidth]{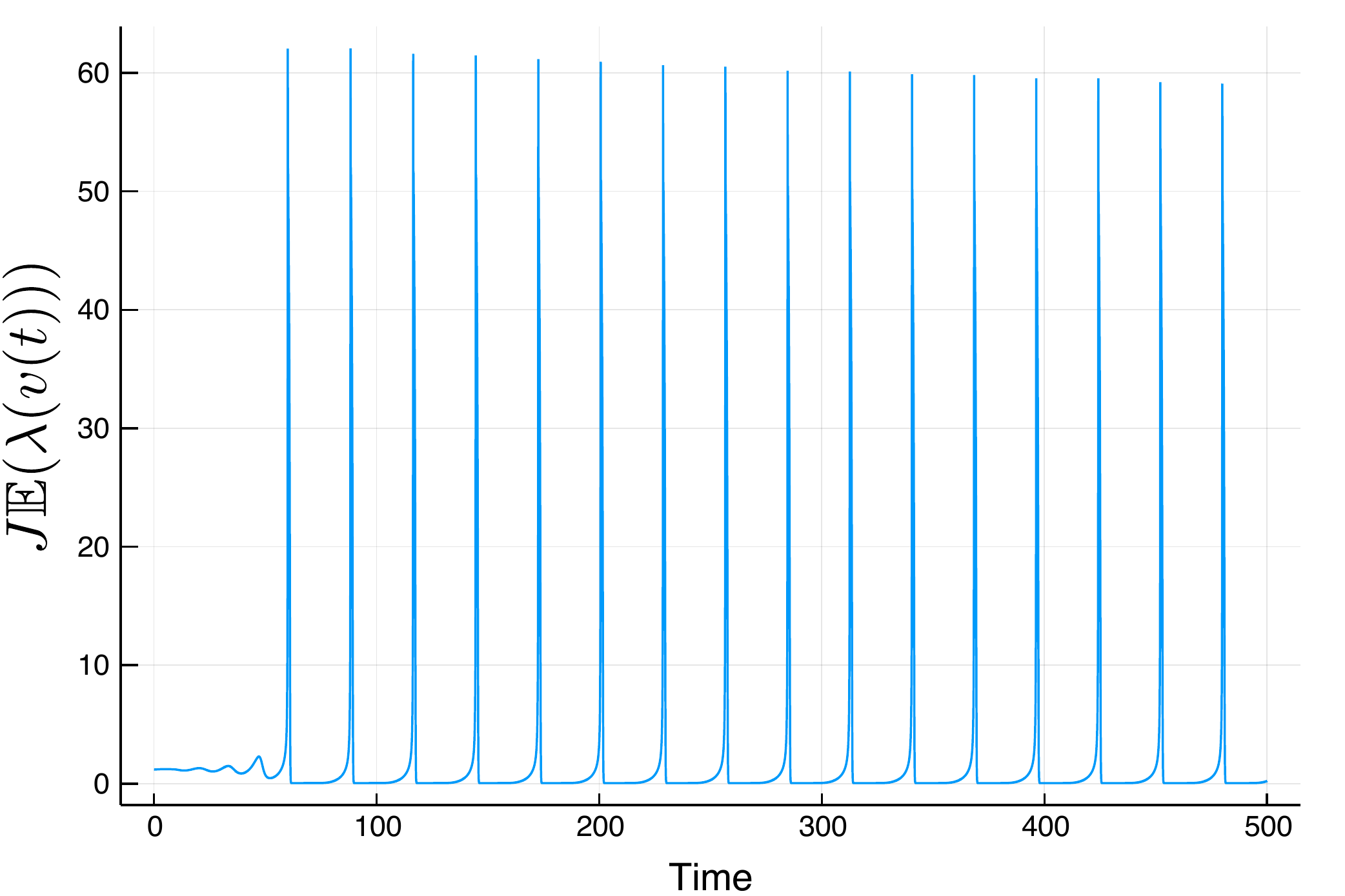}
  \end{center}
  \caption{Effect of synaptic coupling $J$.
    Simulations carried out with mean-field approximation \cref{macro}.
    Top: amplitude of signal with respect to coupling parameter J.
    Bottom: Signal for $J=6, 6.2,6.3,6.45,7$.
    A Hopf bifurcation appears around the critical value $J^* \approx 6.15$.
  }
  \label{fig:Hopf}
\end{figure}

\begin{figure}
  \begin{center}
    \includegraphics[width=0.99\textwidth]{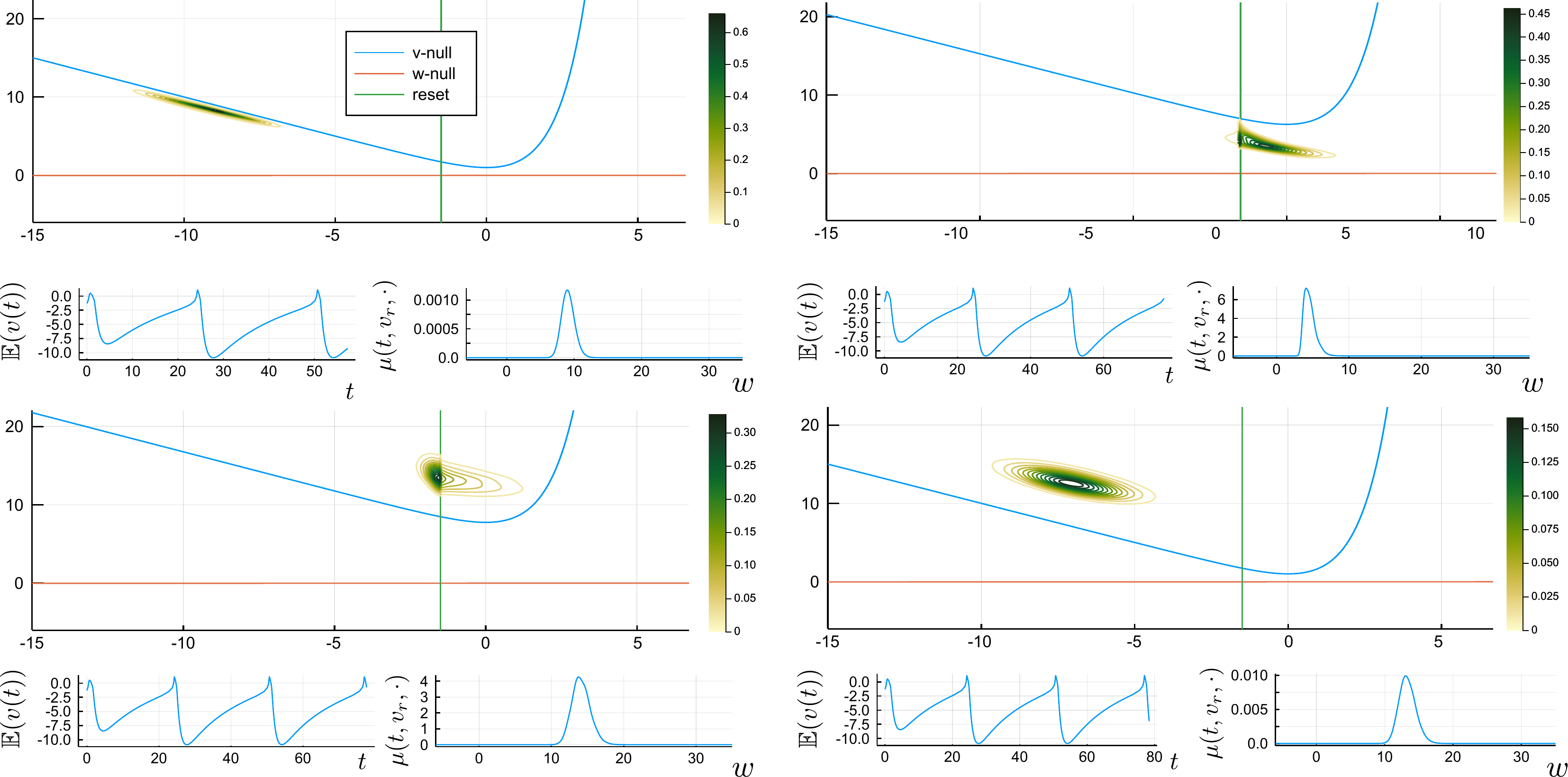}\\
  \end{center}
  \caption{Limit cycle. When the coupling strength is above a critical value, the network synchronised its activity, producing oscillations at macroscopic level (top figure).
    At microscopic level, we see that the density changes the network activity, forcing the $v$-nullcline to rise (middle figures).
    This in return drives the density up in the domain, until crossing the $v$-nullcline and comes back to the bottom of the domain (bottom figures).
    Parameters identical to the ones of the previous figure except for  $J=7$.
  }
  \label{cycle}
\end{figure}

\section{Conclusion}
\label{sec.conclusion}

In this work, we have presented a new nonlinear stochastic model of a network of stochastic spiking neurons.
This model naturally avoids blow-up solutions that may appear using ``threshold crossing'' based spiking like for the Integrate and Fire neuron model.
The system admits a mean-field limit: we have derived the PDE heuristically, and numerical simulations have confirmed the heuristic.
We have shown that this mean-field may be seen as a coupled transport equation, and could be entirely defined by deduction from the solution on the reset interface.

We have designed a Monte Carlo method to simulate the system of neurons.
On the other hand, we have designed a reliable finite volume method in order to simulate the PDE.
This numerical method is overall conservative, theoretically second order accurate in time and preserves positivity.
As further work, it would be interesting to use mesh adaptation.

Taking advantage of those numerical tools, we have studied numerically the network.
As the network size tends to infinity, numerical simulations tend to show propagation of chaos, and convergent behaviour from microscopic description \cref{micro} to macroscopic one \cref{macro}.
Two open theoretical questions remain on this point: does the particle system propagate chaos, and is it possible to rigorously prove the convergence from the particle system to the mean-field model?

By varying the strength of the connectivity, we have observed a Hopf bifurcation, signature of a synchronisation of the activity within the system.
This work may be pursued in several ways: prove	the existence and uniqueness of invariant distributions, and the theoretical and numerical study of bifurcations.

\section*{Acknowledgements}

This research has received funding from the European Un\-ion’s Horizon 2020 Framework Programme for Research and Innovation under the Specific Grant Agreement No. 785907 (Human Brain Project SGA2).

We would like to thank the Julia community for their help online. We would like to thank Laurent Monasse for fruitful discussions.

\bibliographystyle{siamplain}

\bibliography{./articleMMS.bib}

\end{document}